\documentclass{amsart}

\usepackage[latin1]{inputenc}
\usepackage{amsmath,amssymb,graphicx,setspace,verbatim}
\usepackage{color}
\usepackage{comment}
\usepackage{appendix}
\theoremstyle{comment}
\usepackage[color,matrix,arrow]{xy}

\newtheorem*{mcomment}{\color{cyan}{Comment}}

\DeclareMathOperator{\Sp}{Sp}

\DeclareMathOperator{\Orth}{O}
\newcommand{\Fk}{\Bbb{F}_{2^k}}

\input xy
\xyoption{all}

\newtheorem{theorem}{Theorem}[section]

\newtheorem{proposition}[theorem]{Proposition}
\theoremstyle{definition}

\begin{document}

\title{String C-group representations of alternating groups}

\author{Maria Elisa Fernandes}
\address{Maria Elisa Fernandes, Department of Mathematics,
University of Aveiro,
Aveiro,
Portugal}
\email{maria.elisa@ua.pt}

\author{Dimitri Leemans}
\address{Dimitri Leemans, Universit\'{e} Libre de Bruxelles,
D\'{e}partement de Math\'{e}matique,
C.P.216 Alg\`ebre et Combinatoire,
Bld du Triomphe, 1050 Bruxelles,
Belgium}
\email{dleemans@ulb.ac.be}

\maketitle
\begin{abstract}
We prove that for any integer $n\geq 12$, and for every $r$ in the interval $[3, \ldots, \lfloor (n-1)/2\rfloor]$, the group $A_n$ has a string C-group representation of rank $r$ therefore showing that the only alternating group whose set of ranks is not an interval is $A_{11}$.
\end{abstract}


\noindent \textbf{2000 Math Subj. Class:} 52B11, 20D06.

\section{Introduction}
String C-group representations have gained much attention in recent years as they are in one-to-one correspondence with abstract regular polytopes.
More precisely, it is known that string C-groups are automorphism
groups of abstract regular polytopes and that, given an abstract regular polytope and a base flag of the polytope, one can construct a string C-group whose group $G$ is the automorphism group of the polytope~\cite[Section 2E]{arp}. Hence the study of string C-groups has interest not only for group theory, but also for geometry.  

Classifications of string C-groups from almost simple groups started with experimental work of Leemans and Vauthier~\cite{atlasl} and also Hartley~\cite{Halg}. These were pushed further in~\cite{HHalg, LMalg, CLM2012}.
The results obtained in~\cite{atlasl} quickly led to the determination of the highest rank of a string C-group representation for Suzuki groups~\cite{Leemans:2006}. Other families of almost simple groups were then investigated: the almost simple groups with socle $\mathrm{PSL}(2,q)$~\cite{ls07,ls09,DiJuTho}, groups $\mathrm{PSL}(3,q)$ and $\mathrm{PGL}(3,q)$~\cite{Brooksbank:2010}, groups $\mathrm{PSL}(4,q)$~\cite{Brooksbank:2015}, small Ree groups~\cite{Leemans:2015},
orthogonal and symplectic groups in characteristic 2,
 and finally, symmetric groups~\cite{fl} and alternating groups~\cite{flm,flm2}. In particular, only the last four families gave rise to string C-group representations of arbitrary large rank. 
 In~\cite{Brooksbank:2018}, it is shown that, for all integers $m\geq 2$, and all integers $k\geq 2$, 
the orthogonal groups $\Orth^{\pm}(2m,\Fk)$ act on abstract regular polytopes 
of rank $2m$, and the symplectic groups $\Sp(2m,\Fk)$ act on abstract regular polytopes of rank $2m+1$.
A symmetric group $S_n$ is known to have string C-group representations of highest rank $n-1$~\cite{CC} and an alternating group $A_n$ is known to have string C-group representations of highest rank $\lfloor\frac{n-1}{2}\rfloor$ when $n\geq 12$~\cite{an}.

The authors looked at the symmetric groups in~\cite{fl} and showed three important results. Firstly, when $n\geq 5$, the $(n-1)$-simplex is, up to isomorphism, the unique string C-group representation of rank $n-1$ for $S_n$. 
Secondly, they showed that when $n\geq 7$, there is also, up to isomorphism, a unique string C-group representation of rank $n-2$.
And finally, they showed that for every $n\geq 4$, and for every integer $r$ in the interval $[3,\ldots n-1]$, a symmetric group $S_n$ has at least one string C-group representation of rank $r$. Therefore, the symmetric groups have no gaps in their set of ranks. The first and second theorem have been extended in~\cite{extension} where the authors, together with Mark Mixer, classified string C-group representations of rank $n-3$ (for $n\geq 9$) and $n-4$ (for $n\geq 11$) for the symmetric group $S_n$.

Also with Mixer, the authors produced in~\cite{flm,flm2} string C-group representations of rank $\lfloor (n-1)/2 \rfloor$ for the alternating groups, with $n\geq 12$. 
In the process of obtaining these results, they computed all string C-group representations of $A_n$ with $n\leq 12$. They found that the set of ranks for the alternating groups of small degree were as given in Table~\ref{ranksAlt}.
\begin{table}
\begin{center}
\begin{tabular}{||c|c||}
\hline
Group&Set of ranks\\
\hline
$A_5$&\{3\}\\
$A_6$&$\emptyset$\\
$A_7$&$\emptyset$\\
$A_8$&$\emptyset$\\
$A_9$&\{3,4\}\\
$A_{10}$&\{3,4,5\}\\
$A_{11}$&\{3,6\}\\
$A_{12}$&\{3,4,5\}\\
\hline
\end{tabular}
\caption{Set of ranks for small alternating groups.}\label{ranksAlt}
\end{center}
\end{table}
The case $n=11$ turned out to be special in the sense that it was the only example encountered so far of a group whose set of ranks presented gaps. In this paper, we prove a similar result as the third theorem of~\cite{fl}. Our main result is stated as follows.

\begin{theorem}\label{mainT}
For $n\geq 12$ and  for every $3 \leq r \leq \lfloor (n-1)/2 \rfloor$, the group $A_n$ has at least one string C-group representation of rank $r$.
\end{theorem}

Mark Mixer talked about a similar result in 2015 at the AMS Fall Eastern Sectional Meeting in Rutgers (talk 1115-20-283) but never circulated nor published any proof of the result he then announced. We decided it was time to fill in this gap and we give in this paper a constructive proof by explicitly providing for each $n\geq 12$ and each rank $3\leq r \leq \lfloor (n-1)/2\rfloor$ a string C-group representation of rank $r$ for $A_n$.

This theorem shows indeed that the case $n=11$ is special among the alternating groups.
The main tool in the proof of our main theorem is to find good permutation representation graphs that turn out to be CPR graphs, for every rank $3\leq r\leq n$ once $n$ is fixed. We use a proof similar to that of the third theorem of~\cite{fl} to tackle most cases and are just left dealing with finding string C-group representations of rank four and five for $A_n$ when $n$ is even, and ranks 4, 5 and 6 when $n\equiv 3 \mod 4$.

The paper is organised as follows.
In Section~\ref{scg}, we recall the basic definitions about string C-groups.
In Section~\ref{sectioncpr}, we recall the definitions of permutation representation graphs and CPR-graphs and give some results that will be useful in proving Theorem~\ref{mainT}.
Finally, in Section~\ref{proofMainT}, we prove Theorem~\ref{mainT}.

\section{String C-groups}\label{scg}

An abstract polytope is a combinatorial object which generalizes a classical convex polytope in Euclidean space.
When the automorphism group of an abstract polytope acts regularly on its set of flags, the polytope is called \emph{regular}, and in that case, its automorphism group can be presented as what is known as a string C-group.  Additionally, each abstract regular polytope can be constructed from a string C-group, and thus abstract regular polytopes and string C-groups are basically the same objects.  For more details on the subject see \cite[Section 2E]{arp}. 


A \emph{Coxeter group} is a group with generators $\rho_0,\ldots,\rho_{r-1}$ and presentation
$$\langle \rho_i \ | \ (\rho_i\rho_j)^{m_{i,j}}=\varepsilon\mbox{ for all }i,\,j\in\{0,\ldots,r-1\} \rangle$$
where each $m_{i,j}$ is a positive integer or infinity, $m_{i,i}=1$, and $m_{i,j}=m_{j,i}>1$ for $i\neq j$.  It follows from the definition, that a Coxeter group satisfies the next condition called the \emph{intersection property}.
\[\forall J, K \subseteq \{0,\ldots,r-1\}, \langle \rho_j \mid j \in J\rangle \cap \langle \rho_k \mid k \in K\rangle = \langle \rho_j \mid j \in J\cap K\rangle\]
A Coxeter group $G$ can be represented by a \emph{Coxeter diagram} $\mathcal{D}$. This Coxeter diagram $\mathcal{D}$ is a labelled graph which represents the set of relations of $G$. More precisely, the vertices of the graph correspond to the generators $\rho_i$ of $G$, and for each $i$ and $j$, an edge with label $m_{i,j}$ joins the $i$th and the $j$th vertices; conventionally, edges of label 2 are omitted. By a \emph{string (Coxeter) diagram} we mean a Coxeter diagram with each connected component linear. A Coxeter group with a string diagram is called a \emph{string Coxeter group}.

More generally, we define a \emph{string group generated by involutions}, or \emph{sggi} for short, as a pair $(G,S)$ where $G$ is a group, $S:= \{\rho_0,\ldots,\rho_{r-1}\}$ is a finite set of involutions of $G$ that generate $G$ and that satisfy the following property, called the \emph{commuting property}.
\[\forall i,j\in\{0,\ldots, r-1\}, \;|i-j|>1\Rightarrow (\rho_i\rho_j)^2=1\]
Finally, a {\em string C-group representation} of a group $G$ is a pair $(G,S)$ that is a sggi and that satisfies the intersection property. In this case the underlying ``Coxeter" diagram for $(G,S)$ is a string diagram.
The \emph{(Schl\"afli) type} of $(G,S)$ is $\{p_1,\,\ldots,\,p_{r-1}\}$ where $p_i$ is the order of $\rho_{i-1}\rho_i,\,i\in\{1,\ldots,r-1\}$,
and the \emph{rank} of a string C-group representation (or of a sggi) $(G,S)$ is the size of $S$.
When the context is clear, we sometimes omit to precise the set of generators $S$ and we talk about a string C-group $G$ instead of a string C-group representation $(G,S)$.

The {\em set of ranks} of a group $G$ is the largest set of integers $I$ such that for each $r\in I$, there exists at least one string C-group representation of rank $r$ for $G$.

Let $(G,S)$ be a sggi with $S:=\{\rho_0,\,\ldots,\,\rho_{r-1}\}$. We denote by $G_I$ with $I\subseteq \{0,\ldots,\,r-1\}$ the subgroup of $G$ generated by the involutions with indices that are in $I$; each $G_I$ is itself a string C-group.  Also, for $i,j \in  \{0,\ldots,\,r-1\}$, we denote $G_i=\langle \rho_j\,|\, j\neq i\rangle$ and $G_{i,j}:=(G_i)_j$.
The following two results show that when $G_0$ and $G_{r-1}$ are string C-group representations, the intersection property for $(G,S)$ is verified by checking only one condition.

\begin{proposition}\cite[Proposition 2E16]{arp}\label{arp}
Let $(G,S)$ be a sggi with $S:=\{\rho_0,\ldots,\rho_{r-1}\}$.
Suppose that its subgroups, $G_0$ and $G_{r-1}$, are string $C$-group. If $G_0\cap G_{r-1} \cong G_{0,r-1}$, then $(G,S)$ is a string $C$-group.
\end{proposition}

We point out that the inclusion $G_0\cap G_{r-1} \geq G_{0,r-1}$ is trivial, and thus we only need to check that $G_0\cap G_{r-1} \leq G_{0,r-1}$. The following proposition makes it even simpler to check if a pair $(G,S)$ is a string C-group representation when $G_{0,r-1}$ is a maximal subgroup of either $G_0$ or $G_{r-1}$ (or both).

\begin{proposition}
\label{max}\cite[Lemma 2.2]{flm}
Let $\Gamma = \langle \rho_0,\ldots,\rho_{r-1} \rangle$ be a sggi and suppose that its subgroups, $\Gamma_0$ and $\Gamma_{r-1}$, are string $C$-groups. If $\rho_{r-1} \not \in \Gamma_{r-1}$ and $\Gamma_{0,r-1}$ is maximal in $\Gamma_0$, then $\Gamma$ is a string C-group.
\end{proposition}

\section{Permutation representation graphs and CPR graphs}\label{sectioncpr}

Let $G$ be a group of permutations acting on a set $\{1,\,\ldots,\,n\}$.
Let $S:=\{\rho_0,\ldots,\rho_{r-1}\}$ be a set of $r$ involutions of $G$ that generate $G$.
We define $\mathcal{G}$ as the $r$-edge-labeled multigraph  with $n$ vertices and with an $i$-edge $\{a,\,b\}$ whenever $a\rho_i=b$ with $a\neq b$. 

The pair $(G,S)$ is a sggi if and only if $\mathcal{G}$ satisfies the following properties:
\begin{enumerate}
\item The graph induced by edges of label $i$ is a matching;
\item Each connected component of the graph induced by edges of labels $i$ and $j$, for $|i-j| \geq 2$, is a single vertex, a single edge, a double edge, or a square with alternating labels.
\end{enumerate}

When $(G,S)$ is a string C-group, the multigraph $\mathcal{G}$ is called a \emph{CPR graph}, as defined in \cite{pcpr}.  In rank $3$, there are a couple of known results to determine if a $3$-edge-labeled multigraph is a CPR graph.  For higher ranks, no arguments were accomplished.

One simple example of a CPR graph is the one corresponding to the $(n-1)$-simplex as follows:
$$ \xymatrix@-1pc{*+[o][F]{}  \ar@{-}[rr]^0 && *+[o][F]{}  \ar@{-}[rr]^1 && *+[o][F]{}  \ar@{-}[rr]^2 && *+[o][F]{} \ar@{-}[rr]^3 && *+[o][F]{} \ar@{.}[rr] && *+[o][F]{}  \ar@{-}[rr]^{n-2} && *+[o][F]{} \ar@{-}[rr]^{n-1} &&*+[o][F]{} }$$

In~\cite{fl}, for each rank $3\leq r\leq n-2$, a regular $r$-polytope with automorphism group $S_n$ was found. In~\cite{flm}, the authors constructed a regular $r$-polytope, for each rank $r\geq 4$, with automorphism group $A_n$ for some $n$.  This is summarized in the following two theorems, and the associated CPR graphs are given in Table~\ref{TT}.
\begin{theorem}\label{Sym}\cite[Theorem 3]{fl}
For $n\geq 5$ and $3\leq r\leq n-2$, there is an abstract regular $r$-polytope of type $\{n-r+2,\,6,\,3^{r-3}\}$ with automorphism group isomorphic to $S_n$.
\end{theorem}

\begin{theorem}\cite[Theorem 1.1]{flm}
\label{maintheorem}
For each rank $k \geq 3$, there is a regular $k$-polytope with automorphism group isomorphic to an alternating group $A_n$ for some $n$.  In particular, for each even rank $r \geq 4$, there is a regular polytope of type $\{10,3^{r-2}\}$ with automorphism group isomorphic to $A_{2r+1}$, and for each odd rank $q \geq 5$, there is a regular polytope of type $\{10,3^{q-4},6,4\}$ with automorphism group isomorphic to $A_{2q+3}$.
\end{theorem}

\begin{table}
\begin{small}
\begin{center}
\begin{tabular}{||c|c|c||}
\hline
Group&Schl\"afli Type&CPR Graph\\
\hline
\begin{tabular}{c}
\\[-10pt]
$S_{n}$\\[5pt] $^{(3\leq r\leq n-2)}$
\end{tabular} &$\{n-r+2,\,6,\,3^{r-3}\}$&$ \xymatrix@-1.7pc{*+[o][F]{}  \ar@{-}[rr]^0 && *+[o][F]{}  \ar@{-}[rr]^1 && *+[o][F]{}  \ar@{.}[rr] && *+[o][F]{}  \ar@{-}[rr]^0 && *+[o][F]{} \ar@{-}[rr]^{1} && *+[o][F]{}  \ar@{-}[rr]^{2} && *+[o][F]{} \ar@{-}[rr]^3 && *+[o][F]{} \ar@{.}[rr] && *+[o][F]{}  \ar@{-}[rr]^{r-2} && *+[o][F]{} \ar@{-}[rr]^{r-1} &&*+[o][F]{} }$\\
\hline
\begin{tabular}{c}
\\[-10pt]
$A_{2r+1}$\\[5pt] $^{(r\textrm{ even and }\geq 4)}$
\end{tabular} & $\{10,3^{r-2}\}$
&
$\xymatrix@-1.7pc{&& *+[o][F]{}  \ar@{-}[rr]^1 && *+[o][F]{}  \ar@{-}[rr]^2 && *+[o][F]{}  \ar@{-}[rr]^3 && *+[o][F]{}   \ar@{.}[rr] && *+[o][F]{}  \ar@{-}[rr]^{r-2} && *+[o][F]{}  \ar@{-}[rr]^{r-1} && *+[o][F]{} \\
&& && && && && && && \\
 *+[o][F]{}  \ar@{-}[rr]_0 && *+[o][F]{}  \ar@{-}[rr]_1 && *+[o][F]{}   \ar@{-}[uu]^0   \ar@{-}[rr]_2 && *+[o][F]{}  \ar@{-}[uu]^0   \ar@{-}[rr]_3 && *+[o][F]{}  \ar@{-}[uu]^0     \ar@{.}[rr] && *+[o][F]{}  \ar@{-}[uu]_0   \ar@{-}[rr]_{r-2} && *+[o][F]{}  \ar@{-}[uu]_0   \ar@{-}[rr]_{r-1} &&  *+[o][F]{} \ar@{-}[uu]_0     }$\\
\hline
\begin{tabular}{c}
\\[-10pt]
$A_{2r+3}$\\[5pt] $^{(r \textrm{ odd and } \geq 5)}$
\end{tabular} & $\{10,3^{r-4},6,4\}$
&
$ \xymatrix@-1.7pc{&& *+[o][F]{}  \ar@{-}[rr]^1 && *+[o][F]{}  \ar@{-}[rr]^2 && *+[o][F]{}  \ar@{-}[rr]^3 && *+[o][F]{}   \ar@{.}[rr] && *+[o][F]{}  \ar@{-}[rr]^{r-2} && *+[o][F]{}  \ar@{-}[rr]^{r-1} && *+[o][F]{}   \ar@{-}[rr]^{r-2} && *+[o][F]{}\\
&& && && && && && && &&\\
 *+[o][F]{}  \ar@{-}[rr]_0 && *+[o][F]{}  \ar@{-}[rr]_1 && *+[o][F]{}   \ar@{-}[uu]^0   \ar@{-}[rr]_2 && *+[o][F]{}  \ar@{-}[uu]^0   \ar@{-}[rr]_3 && *+[o][F]{}  \ar@{-}[uu]^0     \ar@{.}[rr] && *+[o][F]{}  \ar@{-}[uu]_0   \ar@{-}[rr]_{r-2} && *+[o][F]{}  \ar@{-}[uu]_0   \ar@{-}[rr]_{r-1} &&  *+[o][F]{} \ar@{-}[uu]_0  \ar@{-}[rr]_{r-2} && *+[o][F]{}   \ar@{-}[uu]_0  }$\\
 \hline
\end{tabular}
\caption{String C-groups for $S_n$ and $A_n$}\label{TT}
\end{center}
\end{small}
\end{table}

Permutation representation graphs are a very useful tool for the construction of string groups generated by involutions.  We will use them in the proof of our main theorem.

The term sesqui-extension was first introduced in \cite{flm}. Let us recall its meaning.
Let $\Phi = \langle \alpha_0,\ldots,\alpha_{d-1} \rangle$ be a sggi, and let $\tau$ be an involution in a supergroup of $\Phi$ such that $\tau \not \in \Phi$ and $\tau$ commutes with all of $\Phi $.  For fixed $k$, we define the group $\Phi^*= \langle \alpha_i \tau^{\eta_i}\,|\, i\in \{0,\,\ldots,\,d-1\} \rangle$ where $\eta_i = 1$ if $i=k$ and 0 otherwise, the {\it sesqui-extension} of $\Phi $ with respect to $\alpha_k$ and $\tau$.

\begin{proposition} \label{sesqui}\cite[Proposition 5.4]{flm2}
If $\Phi =\left<\alpha_i\,|\, i=0,\,\ldots,\,d-1\right>$ and $\Phi^*=\langle \alpha_i \tau^{\eta_i}\,|\, i\in \{0,\,\ldots,\,d-1\} \rangle$ is a sesqui-extension of $\Phi$ with respect to $\alpha_k$, then:
\begin{enumerate}
\item $ \Phi^* \cong \Phi$ or $\Phi^* \cong \Phi\times \langle \tau\rangle\cong\Phi\times 2$.
\item whenever $\tau \notin \Phi^*$, $\Phi$ is a string C-group if and only if $\Phi^*$ is a string C-group.
\end{enumerate}
\end{proposition}

Sesqui-extensions will be used later to check the intersection condition on the permutation representations of the groups of our main theorem. We also apply the techniques used in the proof of Theorem~\ref{Sym} based on a construction of Hartley and Leemans available in~\cite{HL}. The key of the proof of Theorem~\ref{Sym} was to start from the CPR graph of the $(n-1)$-simplex with generators $\rho_1, \ldots, \rho_{n-1}$ where $\rho_i$ is the transposition $(i,i+1)$ in $S_n$. Let $d=n-1$. At each step, we start with a string C-group representation of rank $d$ and generators $\rho_1, \ldots, \rho_d$. We replace $\rho_{d-2}$ by $\rho_{d-2}\rho_{d}$ and we drop $\rho_{d}$. As proved in~\cite{fl}, we get in this way a new string C-group with generators $\rho_1, \ldots, \rho_{d-1}$.
We can repeat this until $d=3$.
We give in Table~\ref{sym7} an example of this process for $S_7$.
\begin{table}
\begin{center}
\begin{tabular}{||c|c|c||}
\hline
Generators&CPR graph&Schl\"afli type\\
\hline
(1,2),(2,3),(3,4),(4,5),(5,6),(6,7)&
$ \xymatrix@-1.7pc{*+[o][F]{}  \ar@{-}[rr]^1 && *+[o][F]{}  \ar@{-}[rr]^2 && *+[o][F]{} \ar@{-}[rr]^{3} && *+[o][F]{}  \ar@{-}[rr]^{4} && *+[o][F]{} \ar@{-}[rr]^5 && *+[o][F]{} \ar@{-}[rr]^{6} && *+[o][F]{}  }$
&\{3,3,3,3,3\}\\
(1,2),(2,3),(3,4),(4,5)(6,7),(5,6)&
$ \xymatrix@-1.7pc{*+[o][F]{}  \ar@{-}[rr]^1 && *+[o][F]{}  \ar@{-}[rr]^2 && *+[o][F]{} \ar@{-}[rr]^{3} && *+[o][F]{}  \ar@{-}[rr]^{4} && *+[o][F]{} \ar@{-}[rr]^5 && *+[o][F]{} \ar@{-}[rr]^{4} && *+[o][F]{}  }$
&\{3,3,6,4\}\\
(1,2),(2,3),(3,4)(5,6),(4,5)(6,7)&
$ \xymatrix@-1.7pc{*+[o][F]{}  \ar@{-}[rr]^1 && *+[o][F]{}  \ar@{-}[rr]^2 && *+[o][F]{} \ar@{-}[rr]^{3} && *+[o][F]{}  \ar@{-}[rr]^{4} && *+[o][F]{} \ar@{-}[rr]^3 && *+[o][F]{} \ar@{-}[rr]^{4} && *+[o][F]{}  }$
&\{3,6,5\}\\
(1,2),(2,3)(4,5)(6,7),(3,4)(5,6)&
$ \xymatrix@-1.7pc{*+[o][F]{}  \ar@{-}[rr]^1 && *+[o][F]{}  \ar@{-}[rr]^2 && *+[o][F]{} \ar@{-}[rr]^{3} && *+[o][F]{}  \ar@{-}[rr]^{2} && *+[o][F]{} \ar@{-}[rr]^3 && *+[o][F]{} \ar@{-}[rr]^{2} && *+[o][F]{}  }$
&\{6,6\}\\
\hline
\end{tabular}
\caption{The induction process used on $S_7$}\label{sym7}
\end{center}
\end{table}

In order to prove that the permutation groups of our main theorem are isomorphic to alternating groups we use the following results.

\begin{theorem}\cite{GJ} \label{GJ} Let $G$ be a primitive permutation group of finite degree $n$, containing a
cycle of prime length fixing at least three points. Then $G\geq A_n$.
\end{theorem}

\begin{proposition}\label{0101tail}\cite[Proposition 3.3]{flm2} Let $G=\langle \rho_0,\ldots,\rho_{r-1}\rangle$ be a transitive permutation group acting on the points $\{1,\ldots, n\}$ with $n\geq 5$, and let 
$G^*=\langle\rho_0,\ldots,\rho_{r-1}, \rho_r, \rho_{r+1}\rangle$, where
\begin{center}
\begin{tabular}{l}
$\rho_r = (i, n+1)(n+2, n+3)$ for some $i\in\{1,\ldots, n\}$\\
$\rho_{r+1}=(n+1, n+2)(n+3, n+4)$.
\end{tabular}
\end{center}
Then $G^*$ is isomorphic to $S_{n+4}$ if it contains an odd permutation, and to $A_{n+4}$ otherwise.
\end{proposition}

\begin{proposition}\label{Sym2(1)}
The following graph, with $n\geq 8$ vertices  and $r\in\{3,\ldots, \frac{n-2}{2}\}$, is a CPR graph for $(S_{\frac{n-4}{2}}\times S_{\frac{n+4}{2}})^+$.
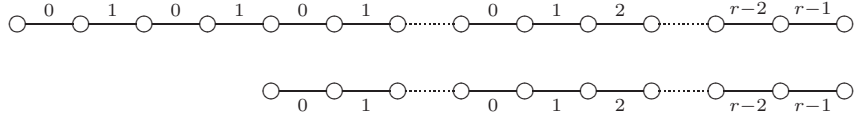
$$\xymatrix@-0.5pc{*+[o][F]{}  \ar@{-}[r]^0& *+[o][F]{}  \ar@{-}[r]^1&*+[o][F]{} \ar@{-}[r]^0& *+[o][F]{}  \ar@{-}[r]^1 & *+[o][F]{}  \ar@{-}[r]^0&*+[o][F]{}  \ar@{-}[r]^1&*+[o][F]{}   \ar@{.}[r] & *+[o][F]{}  \ar@{-}[r]^0 & *+[o][F]{}  \ar@{-}[r]^1 & *+[o][F]{}  \ar@{-}[r]^2 & *+[o][F]{} \ar@{.}[r] & *+[o][F]{_{}}  \ar@{-}[r]^{r-2}& *+[o][F]{_{}}  \ar@{-}[r]^{r-1}  & *+[o][F]{} \\
& & &&*+[o][F]{}  \ar@{-}[r]_0 & *+[o][F]{}  \ar@{-}[r]_1 &*+[o][F]{} \ar@{.}[r] & *+[o][F]{}  \ar@{-}[r]_0& *+[o][F]{}  \ar@{-}[r]_1 & *+[o][F]{}  \ar@{-}[r]_2 & *+[o][F]{}  \ar@{.}[r] & *+[o][F]{}  \ar@{-}[r]_{r-2} & *+[o][F]{}  \ar@{-}[r]_{r-1} & *+[o][F]{}} $$
\end{proposition}
\begin{proof}
Let $G$ be the group with the permutation representation given by the graph of this proposition.
Let us first consider $r=3$.
$$
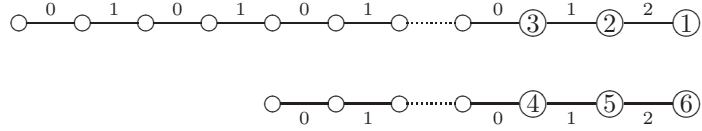
\xymatrix@-0.5pc{*+[o][F]{}  \ar@{-}[r]^0& *+[o][F]{}  \ar@{-}[r]^1&*+[o][F]{} \ar@{-}[r]^0& *+[o][F]{}  \ar@{-}[r]^1 & *+[o][F]{}  \ar@{-}[r]^0&*+[o][F]{}  \ar@{-}[r]^1&*+[o][F]{}   \ar@{.}[r] & *+[o][F]{}  \ar@{-}[r]^0 & *+[o][F]{3}  \ar@{-}[r]^1 &   *+[o][F]{2} \ar@{-}[r]^2 &*+[o][F]{1}  \\
& & &&*+[o][F]{}  \ar@{-}[r]_0 & *+[o][F]{}  \ar@{-}[r]_1 &*+[o][F]{} \ar@{.}[r] & *+[o][F]{}  \ar@{-}[r]_0& *+[o][F]{4}  \ar@{-}[r]_1 &*+[o][F]{5}  \ar@{-}[r]_2 &*+[o][F]{6}  }$$

We have that $G_0$  and $G_2$ are string C-groups and  as $G_0\cap G_2=G_{0,2}\cong C_2$, $G$ is itself a string C-group
by Proposition~\ref{arp}.

Let us prove that is is isomorphic to $(S_{\frac{n-4}{2}}\times S_{\frac{n+4}{2}})^+$. 
We first prove that $G$ contains the  3-cycles $(1,2,3)$ and $(4,5,6)$ (the vertices of the above graph on the right). 
Let $l$ be the least integer such that $(\rho_0\rho_1)^l$ fixes all the vertices of the component  of the graph on the bottom. We have that $(\rho_1\rho_2)^2=(1,2,3)(4,5,6)$. The latter element conjugated by $(\rho_0\rho_1)^l$ is equal to $\alpha=(a,b,c)(4,5,6)$ with $\{a,b,c\}\cap\{1,2,3\}=\{1\}$. Hence $(\alpha(\rho_1\rho_2)^2)^5=(4,6,5)$ and $(1,2,3)=(4,6,5)(\rho_1\rho_2)^2$.

Now by transitivity in each of the two components of the graph we have that  $A_{\frac{n-4}{2}}\times A_{\frac{n+4}{2}}$ is a subgroup of $G$. As in addition $\rho_2\notin A_{\frac{n-4}{2}}\times A_{\frac{n+4}{2}}$ and $G$ is an even group we have that $G$ is isomorphic to $(S_{\frac{n-4}{2}}\times S_{\frac{n+4}{2}})^+$. 

Now let $r>3$.
We may assume by induction that $G_{r-1}$ is a string C-group isomorphic to $(S_{\frac{n-6}{2}}\times S_{\frac{n+2}{2}})^+$. 
In addition $G_0$ is a string C-group isomorphic to $S_{r-1}$.  By the intersection of the orbits of $G_0$ and $G_{r-1}$ we conclude that  $G_0\cap G_{r-1}$ and   $G_{0,r-1}$ are both isomorphic to $S_{r-2}$. Therefore $G$ is a string C-group. Moreover it is clearly isomorphic to $(S_{\frac{n-4}{2}}\times S_{\frac{n+4}{2}})^+$. 
\end{proof}

\begin{proposition}\label{Sym2}
The following graph, with $n\geq 10$ vertices  and $r\in\{4,\ldots, \frac{n-2}{2}\}$, is a CPR graph for $S_n$.
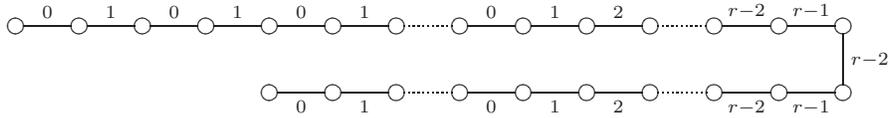
$$\xymatrix@-0.5pc{*+[o][F]{}  \ar@{-}[r]^0& *+[o][F]{}  \ar@{-}[r]^1&*+[o][F]{} \ar@{-}[r]^0& *+[o][F]{}  \ar@{-}[r]^1 & *+[o][F]{}  \ar@{-}[r]^0&*+[o][F]{}  \ar@{-}[r]^1&*+[o][F]{}   \ar@{.}[r] & *+[o][F]{}  \ar@{-}[r]^0 & *+[o][F]{}  \ar@{-}[r]^1 & *+[o][F]{}  \ar@{-}[r]^2 & *+[o][F]{} \ar@{.}[r] & *+[o][F]{_{}}  \ar@{-}[r]^{r-2}& *+[o][F]{_{}}  \ar@{-}[r]^{r-1}  & *+[o][F]{} \\
& & &&*+[o][F]{}  \ar@{-}[r]_0 & *+[o][F]{}  \ar@{-}[r]_1 &*+[o][F]{} \ar@{.}[r] & *+[o][F]{}  \ar@{-}[r]_0& *+[o][F]{}  \ar@{-}[r]_1 & *+[o][F]{}  \ar@{-}[r]_2 & *+[o][F]{}  \ar@{.}[r] & *+[o][F]{}  \ar@{-}[r]_{r-2} & *+[o][F]{}  \ar@{-}[r]_{r-1} & *+[o][F]{} \ar@{-}[u]_{r-2} } $$
\end{proposition}
\begin{proof}
The permutation representation graph is connected, hence $G$, the group with the permutation representation graph given in this lemma, is transitive.  Let $x$ be the first point on the left of the graph. 
The stabilizer of $x$ has at most the same orbits as $G_0$. Consider the vertices $y$ and $z$ as in the following graph.  $$
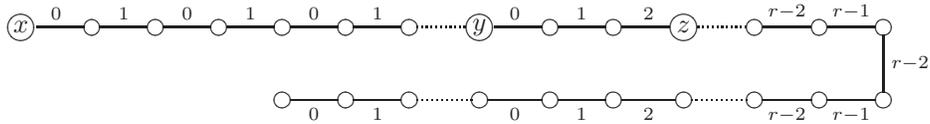
\xymatrix@-0.5pc{*+[o][F]{x}  \ar@{-}[r]^0& *+[o][F]{}  \ar@{-}[r]^1&*+[o][F]{} \ar@{-}[r]^0& *+[o][F]{}  \ar@{-}[r]^1 & *+[o][F]{}  \ar@{-}[r]^0&*+[o][F]{}  \ar@{-}[r]^1&*+[o][F]{}   \ar@{.}[r] & *+[o][F]{y}  \ar@{-}[r]^0 & *+[o][F]{}  \ar@{-}[r]^1 & *+[o][F]{}  \ar@{-}[r]^2 & *+[o][F]{z} \ar@{.}[r] & *+[o][F]{_{}}  \ar@{-}[r]^{r-2}& *+[o][F]{_{}}  \ar@{-}[r]^{r-1}  & *+[o][F]{} \\
& & &&*+[o][F]{}  \ar@{-}[r]_0 & *+[o][F]{}  \ar@{-}[r]_1 &*+[o][F]{} \ar@{.}[r] & *+[o][F]{}  \ar@{-}[r]_0& *+[o][F]{}  \ar@{-}[r]_1 & *+[o][F]{}  \ar@{-}[r]_2 & *+[o][F]{}  \ar@{.}[r] & *+[o][F]{}  \ar@{-}[r]_{r-2} & *+[o][F]{}  \ar@{-}[r]_{r-1} & *+[o][F]{} \ar@{-}[u]_{r-2} } $$
We have that $y\rho_2^{\rho_1\rho_0}=z$ and $\rho_2^{\rho_1\rho_0}$ fixes $x$.  More generally the appropriate conjugations of $\rho_2$ by powers of $\rho_0\rho_1$ fuse the orbits of $G_0$ while fixing $x$.
Hence $G$ is 2-transitive and therefore primitive. Moreover, it contains a 3-cycle (explicitly  given in the proof of Proposition~\ref{Sym2(1)}) and an odd permutation. Hence, by Theorem~\ref{GJ}, it is isomorphic to $S_{n-1}$.
By Proposition~\ref{sesqui} and~\cite[Table 2]{flm2} we may conclude that $G_0$ is a string C-group  isomorphic to $2\times (2\wr S_{r-1})$. 
By Proposition~\ref{Sym2(1)}, the group $G_{r-1}$ is a string C-group isomorphic to $(S_{\frac{n-6}{2}}\times S_{\frac{n+2}{2}})^+$.
From the intersection of the orbits of $G_0$ and $G_{r-1}$ we also conclude that $G_0\cap G_{r-1}=G_{0,r-1}\cong 2\times(S_{\frac{n-7}{2}}\times S_{\frac{n+1}{2}})^+$.
Hence $G$ is a string C-group.
\end{proof}

\begin{proposition}\label{Sym3}
The following graph, with $n\geq 8$ (and $r=n/2$) vertices is a CPR graph for $S_n$.
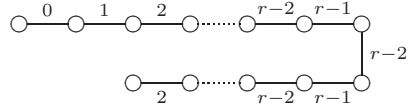
$$\xymatrix@-0.7pc{ *+[o][F]{}  \ar@{-}[r]^0 &*+[o][F]{}  \ar@{-}[r]^1 & *+[o][F]{} \ar@{-}[r]^2 & *+[o][F]{} \ar@{.}[r] & *+[o][F]{}  \ar@{-}[r]^{r-2}& *+[o][F]{}  \ar@{-}[r]^{r-1}  & *+[o][F]{}    \\
 & &*+[o][F]{}  \ar@{-}[r]_2 & *+[o][F]{}  \ar@{.}[r] & *+[o][F]{}  \ar@{-}[r]_{r-2} & *+[o][F]{}  \ar@{-}[r]_{r-1} & *+[o][F]{}  \ar@{}\ar@{-}[u]_{r-2}  } $$
\end{proposition}
\begin{proof}
Let $G$ be the group with the permutation representation graph given in this proposition.
Removing the $0$-edge from the graph we get a CPR graph for a symmetric group of degree $n-1$ (see Table 2 of \cite{flm2}). Hence $G_0$ is a string C-group.
Now consider the group $H$ with the following permutation representation graph.
$$
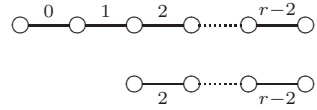
\xymatrix@-0.7pc{ *+[o][F]{}  \ar@{-}[r]^0 &*+[o][F]{}  \ar@{-}[r]^1 & *+[o][F]{} \ar@{-}[r]^2 & *+[o][F]{} \ar@{.}[r] & *+[o][F]{}  \ar@{-}[r]^{r-2}& *+[o][F]{}   \\
 & &*+[o][F]{}  \ar@{-}[r]_2 & *+[o][F]{}  \ar@{.}[r] & *+[o][F]{}  \ar@{-}[r]_{r-2} & *+[o][F]{}  } $$
For $r=4$, $H$ is a string C-group isomorphic to $2\times S_4$. Assume by induction that  $H_{r-2}$ is a string C-group isomorphic to $S_{r-1}\times S_{r-3}$.
As $H_0$ is a string C-group and $H_0\cap H_{r-2}\leq S_{r-2}\times S_{r-3}\cong H_{0,r-2}$, $H$ is a string C-group. Moreover  $H\cong S_{r-1}\times S_{r-3}$. Now by Proposition~\ref{sesqui} the group $G_{r-1}$ is a string C-group isomorphic to $2\times S_{r-1}\times S_{r-3}$. By the intersection of the orbits of $G_0$ and $G_{r-1}$ we have that $G_0\cap G_{r-1}=G_{0,r-1}$ Hence $G$ is a string C-group. 
As $G_0\cong S_{n-1}$ and stabilizes  the first vertex on the left, we have that $G\cong S_n$.
\end{proof}

\begin{proposition}\label{Sym4}
The following graph with $n$ vertices, $n\equiv 3 \mod 4$ and $n\geq 11$, is a CPR graph for $S_n$.
$$\xymatrix@-0.7pc{ *+[o][F]{}  \ar@{-}[r]^0 & *+[o][F]{}  \ar@{-}[r]^1 &*+[o][F]{}  \ar@{=}[r]^0_2 & *+[o][F]{} \ar@{-}[r]^1 &  *+[o][F]{}  \ar@{-}[r]^0 & *+[o][F]{}  \ar@{-}[r]^1 & *+[o][F]{}  \ar@{.}[r] & *+[o][F]{}  \ar@{-}[r]^1 &*+[o][F]{}  \ar@{-}[r]^2 & *+[o][F]{} \ar@{-}[r]^3 & *+[o][F]{}  } $$
\end{proposition}
\begin{proof}
Let $G$ be the group with the permutation representation graph given in this proposition.
The group $G_3$ is an even transitive group containing a 3-cycle, namely $(\rho_1\rho_2)^4$, and the stabilizer of a point on $G_3$ is transitive on the remaining points. Hence by Theorem~\ref{GJ} $G_3\cong A_{n-1}$.
Consequently $G\cong S_n$. Moreover as $G_3$ is a simple group generated by three independent involution, therefore it is a a string C-group. It is also easy to check that $G_0$ is string C-group and that  $G_3\cap G_0=G_{0,3}$, as it is sufficient to consider the case $n=11$. Hence $G$ is a string C-group isomorphic to $S_n$ as wanted.
\end{proof}

\section{Proof of Theorem~\ref{mainT} }\label{proofMainT}

For the rank 3, we can rely on~\cite{Con80,Con81} which covers all but a small number of small cases that can be easily dealt with  {\sc Magma}~\cite{magma}. Another possibility is to use~\cite{SC94}.
Hence we have to construct examples of rank 4 and above. 
Also, the case where $n=12$ is done in~\cite{flm}.

We divide the rest of the proof as follows. First we deal with the case where $n$ is even.
In Theorem~\ref{evenr}, we construct a family of polytopes of rank $6\leq r \leq \frac{n-2}{2}$.
 The rank 4 and 5 for $n$ even are then dealt with in Theorems~\ref{04} and~\ref{05} when $n\equiv 0 \mod 4$ and in Theorems~\ref{24} and~\ref{25} when $n\equiv 2 \mod 4$ respectively.
Then we deal with the case where $n$ is odd and we need to divide the discussion depending on whether $n\equiv 1 \mod 4$ or $n\equiv 3 \mod 4$.
For $n\equiv 1\mod 4$, we give in Theorem~\ref{1r} a family of polytopes of rank $4\leq r \leq \frac{n-1}{2}$.
For $n\equiv 3\mod 4$, we give in Theorem~\ref{3r} a family of polytopes of rank $7\leq r \leq \frac{n-1}{2}$. Then we give, in Theorems~\ref{34},~\ref{35}, and~\ref{36}, families of polytopes of respective ranks 4, 5 and 6 to finish the proof of Theorem~\ref{mainT}

\subsection{The even case}

\begin{theorem}\label{evenr}
Let $n\geq 14$ be an even integer and $6\leq r \leq \frac{n-2}{2}$.
The group $A_n$ admits a string C-group representation of rank $r$, with Schl\"afli type $\{LCM(4+i, i), 6, 3^{r-6},6,6,3\}$ (with $i=  (n-2)/2-r+1)$ and with the following CPR graph for $n\equiv 2\mod 4$

$$
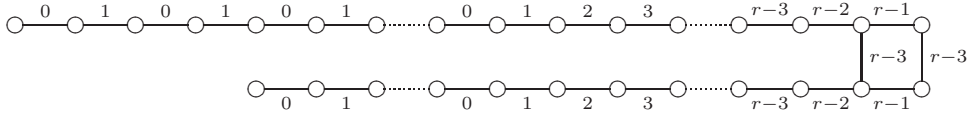
\xymatrix@-0.6pc{*+[o][F]{}  \ar@{-}[r]^0 & *+[o][F]{}  \ar@{-}[r]^1 &*+[o][F]{}  \ar@{-}[r]^0 & *+[o][F]{}  \ar@{-}[r]^1 & *+[o][F]{}  \ar@{-}[r]^0 & *+[o][F]{}  \ar@{-}[r]^1 & *+[o][F]{}  \ar@{.}[r] &*+[o][F]{}  \ar@{-}[r]^0 & *+[o][F]{}  \ar@{-}[r]^1 & *+[o][F]{}  \ar@{-}[r]^2 & *+[o][F]{}  \ar@{-}[r]^3 & *+[o][F]{} \ar@{.}[r] & *+[o][F]{_{}}  \ar@{-}[r]^{r-3}& *+[o][F]{_{}}  \ar@{-}[r]^{r-2}  & *+[o][F]{}  \ar@{-}[r]^{r-1}  & *+[o][F]{}\\
&& & &*+[o][F]{}   \ar@{-}[r]_0&*+[o][F]{}  \ar@{-}[r]_1 & *+[o][F]{}  \ar@{.}[r] &*+[o][F]{} \ar@{-}[r] _0 & *+[o][F]{}  \ar@{-}[r]_1& *+[o][F]{}  \ar@{-}[r]_2 & *+[o][F]{}  \ar@{-}[r]_3 & *+[o][F]{}  \ar@{.}[r] & *+[o][F]{}  \ar@{-}[r]_{r-3} & *+[o][F]{}  \ar@{-}[r]_{r-2} & *+[o][F]{}  \ar@{-}[r]_{r-1}\ar@{-}[u]_{r-3}  & *+[o][F]{}\ar@{-}[u]_{r-3}} $$

and the following CPR-graph for $n\equiv 0\mod 4$.

$$
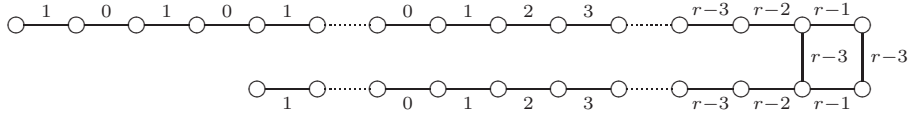
\xymatrix@-0.6pc{*+[o][F]{}  \ar@{-}[r]^1 &*+[o][F]{}  \ar@{-}[r]^0 & *+[o][F]{}  \ar@{-}[r]^1 & *+[o][F]{}  \ar@{-}[r]^0 & *+[o][F]{}  \ar@{-}[r]^1 & *+[o][F]{}  \ar@{.}[r] &*+[o][F]{}  \ar@{-}[r]^0 & *+[o][F]{}  \ar@{-}[r]^1 & *+[o][F]{}  \ar@{-}[r]^2 & *+[o][F]{}  \ar@{-}[r]^3 & *+[o][F]{} \ar@{.}[r] & *+[o][F]{_{}}  \ar@{-}[r]^{r-3}& *+[o][F]{_{}}  \ar@{-}[r]^{r-2}  & *+[o][F]{}  \ar@{-}[r]^{r-1}  & *+[o][F]{}\\
&& & &*+[o][F]{}  \ar@{-}[r]_1 & *+[o][F]{}  \ar@{.}[r] &*+[o][F]{} \ar@{-}[r] _0 & *+[o][F]{}  \ar@{-}[r]_1& *+[o][F]{}  \ar@{-}[r]_2 & *+[o][F]{}  \ar@{-}[r]_3 & *+[o][F]{}  \ar@{.}[r] & *+[o][F]{}  \ar@{-}[r]_{r-3} & *+[o][F]{}  \ar@{-}[r]_{r-2} & *+[o][F]{}  \ar@{-}[r]_{r-1}\ar@{-}[u]_{r-3}  & *+[o][F]{}\ar@{-}[u]_{r-3}} $$
\end{theorem}
\begin{proof}
We consider the following CPR graph which was given as the example of highest possible rank in the case where $n\geq14$ is even and $r=\frac{n-2}{2}\geq 6$  in~\cite{flm2}.
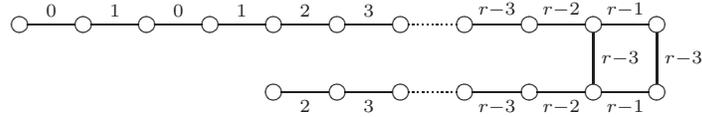
$$\xymatrix@-0.5pc{*+[o][F]{}  \ar@{-}[r]^0 & *+[o][F]{}  \ar@{-}[r]^1 &*+[o][F]{}  \ar@{-}[r]^0 & *+[o][F]{}  \ar@{-}[r]^1 & *+[o][F]{}  \ar@{-}[r]^2 & *+[o][F]{}  \ar@{-}[r]^3 & *+[o][F]{} \ar@{.}[r] & *+[o][F]{_{}}  \ar@{-}[r]^{r-3}& *+[o][F]{_{}}  \ar@{-}[r]^{r-2}  & *+[o][F]{}  \ar@{-}[r]^{r-1}  & *+[o][F]{}\\
& & & & *+[o][F]{}  \ar@{-}[r]_2 & *+[o][F]{}  \ar@{-}[r]_3 & *+[o][F]{}  \ar@{.}[r] & *+[o][F]{}  \ar@{-}[r]_{r-3} & *+[o][F]{}  \ar@{-}[r]_{r-2} & *+[o][F]{}  \ar@{-}[r]_{r-1}\ar@{-}[u]_{r-3}  & *+[o][F]{}\ar@{-}[u]_{r-3}} $$

It gives a string C-group representation of type $\{5,6,3^{r-6},6,6,3\}$. From this graph we construct a family of graphs with $n$ vertices and $r\in\{6,\ldots, \frac{n-2}{2}\}$ adding on the top and on the bottom of the graph above, two sequences of edges, of the same size, with alternate labels $0$ and $1$.
So we have the following two possibilities.

$$
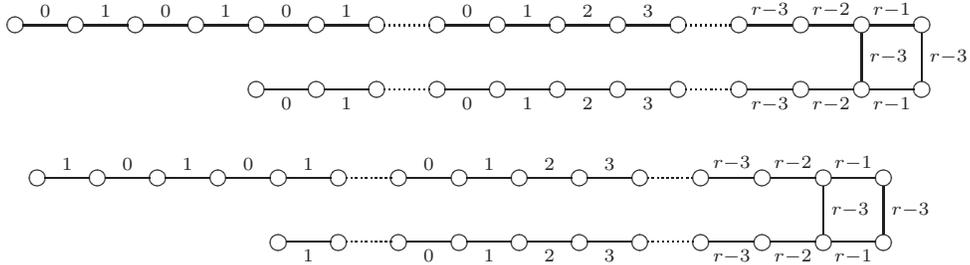
\xymatrix@-0.6pc{*+[o][F]{}  \ar@{-}[r]^0 & *+[o][F]{}  \ar@{-}[r]^1 &*+[o][F]{}  \ar@{-}[r]^0 & *+[o][F]{}  \ar@{-}[r]^1 & *+[o][F]{}  \ar@{-}[r]^0 & *+[o][F]{}  \ar@{-}[r]^1 & *+[o][F]{}  \ar@{.}[r] &*+[o][F]{}  \ar@{-}[r]^0 & *+[o][F]{}  \ar@{-}[r]^1 & *+[o][F]{}  \ar@{-}[r]^2 & *+[o][F]{}  \ar@{-}[r]^3 & *+[o][F]{} \ar@{.}[r] & *+[o][F]{_{}}  \ar@{-}[r]^{r-3}& *+[o][F]{_{}}  \ar@{-}[r]^{r-2}  & *+[o][F]{}  \ar@{-}[r]^{r-1}  & *+[o][F]{}\\
&& & &*+[o][F]{}   \ar@{-}[r]_0&*+[o][F]{}  \ar@{-}[r]_1 & *+[o][F]{}  \ar@{.}[r] &*+[o][F]{} \ar@{-}[r] _0 & *+[o][F]{}  \ar@{-}[r]_1& *+[o][F]{}  \ar@{-}[r]_2 & *+[o][F]{}  \ar@{-}[r]_3 & *+[o][F]{}  \ar@{.}[r] & *+[o][F]{}  \ar@{-}[r]_{r-3} & *+[o][F]{}  \ar@{-}[r]_{r-2} & *+[o][F]{}  \ar@{-}[r]_{r-1}\ar@{-}[u]_{r-3}  & *+[o][F]{}\ar@{-}[u]_{r-3}} $$

$$\xymatrix@-0.6pc{*+[o][F]{}  \ar@{-}[r]^1 &*+[o][F]{}  \ar@{-}[r]^0 & *+[o][F]{}  \ar@{-}[r]^1 & *+[o][F]{}  \ar@{-}[r]^0 & *+[o][F]{}  \ar@{-}[r]^1 & *+[o][F]{}  \ar@{.}[r] &*+[o][F]{}  \ar@{-}[r]^0 & *+[o][F]{}  \ar@{-}[r]^1 & *+[o][F]{}  \ar@{-}[r]^2 & *+[o][F]{}  \ar@{-}[r]^3 & *+[o][F]{} \ar@{.}[r] & *+[o][F]{_{}}  \ar@{-}[r]^{r-3}& *+[o][F]{_{}}  \ar@{-}[r]^{r-2}  & *+[o][F]{}  \ar@{-}[r]^{r-1}  & *+[o][F]{}\\
&& & &*+[o][F]{}  \ar@{-}[r]_1 & *+[o][F]{}  \ar@{.}[r] &*+[o][F]{} \ar@{-}[r] _0 & *+[o][F]{}  \ar@{-}[r]_1& *+[o][F]{}  \ar@{-}[r]_2 & *+[o][F]{}  \ar@{-}[r]_3 & *+[o][F]{}  \ar@{.}[r] & *+[o][F]{}  \ar@{-}[r]_{r-3} & *+[o][F]{}  \ar@{-}[r]_{r-2} & *+[o][F]{}  \ar@{-}[r]_{r-1}\ar@{-}[u]_{r-3}  & *+[o][F]{}\ar@{-}[u]_{r-3}} $$

Let $G$ be a group having one of the permutation representations above.
The group $G_0$ is a subgroup of $2\times(2^{r}:S_r )$, but as $G_0$ is an even group, by Proposition~\ref{sesqui}, $ G_0\cong 2^{r}:S_r$. Moreover $G_0$ is a string C-group. By Propositions~\ref{sesqui} and \ref{Sym2} the group $G_{r-1}$ is a string C-group isomorphic to $S_{n-2}$. As $G_{r-1}$ is a maximal subgroup of $A_n$ and $\rho_{r-1}\not\in G_{r-1}$, it follows from Proposition~\ref{arp} that $G\cong A_n$.
Finally, $G_{0,r-1} = G_0 \cap G_{r-1}$ as the orbits of $G_0\cap G_{r-1}$ have to be suborbits of $G_0$ and of $G_{r-1}$ and $G_{0,r-1}$ is the biggest subgroup we can obtain that would have such orbits. Hence $G$ is isomorphic to $A_n$ and the permutation representation graph above is a CPR-graph.

Let $i=  (n-2)/2-r+1$.  Then it is easy to see from the CPR-graph that the Schl\"afli type of the string C-group of rank $r$ for $A_n$ obtained by this construction is
$\{LCM(4+i, i), 6, 3^{r-6},6,6,3\}$.
The first entry of the symbol comes from the fact that there are 0-1-components on the upper side of the graph and on the lower side of the graph and the upper one has 4 more vertices than the lower one.
\end{proof}
There remains to construct examples in rank 4 and 5 for $n$ even. We split the discussion in two cases, namely the case where $n\equiv 0 \mod 4$ and the case where $n\equiv 2 \mod 4$.
%
%
\begin{theorem}\label{24}
Let $n\equiv 2 \mod 4$ with $n\geq 10$.	
The group $A_n$ admits a string C-group representation of rank $4$, with Schl\"afli type $\{5, 6, n-4\}$, with the following CPR-graph.
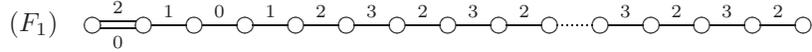
\begin{center}
\begin{tabular}{cc}
$(F_1)$
&$ \xymatrix@-1.7pc{*+[o][F]{} \ar@<.3ex>@{-}[rr]^2 \ar@<-.3ex>@{-}[rr]_0 && *+[o][F]{}  \ar@{-}[rr]^1 && *+[o][F]{}  \ar@{-}[rr]^0 && *+[o][F]{} \ar@{-}[rr]^{1} && *+[o][F]{}  \ar@{-}[rr]^{2} && *+[o][F]{} \ar@{-}[rr]^3 && *+[o][F]{} \ar@{-}[rr]^{2} && *+[o][F]{} \ar@{-}[rr]^{3} && *+[o][F]{} \ar@{-}[rr]^2 && *+[o][F]{} \ar@{.}[rr] && *+[o][F]{} \ar@{-}[rr]^{3} && *+[o][F]{}  \ar@{-}[rr]^{2} && *+[o][F]{} \ar@{-}[rr]^{3} &&*+[o][F]{} \ar@{-}[rr]^{2} &&*+[o][F]{}  }$\\
\end{tabular}
\end{center}
\end{theorem}
\begin{proof}
Let $G$ be a group having one of the permutation representation graphs above.
In this case, $G_3 \cong C_2\times A_5$ is obviously a string C-group of rank 3.
Moreover, $G_{0,3}\cong C_2\times D_{6} \cong D_{12}$ and therefore $G_{0,3}$ is maximal in $G_3$.
So, by Proposition~\ref{max}, 
it remains to prove that $G_0$ is also a string C-group.
Now, $G_{0,3}$ and $G_{0,1}$ are obviously string C-groups as they are dihedral groups.
The group $G_{0,1,3} \cong C_2$ and the subgroups $G_{0,3}$ and $G_{0,1}$ will have the same intersection no matter what the value of $n$ is. We can thus assume $n=10$ and check by hand or using {\sc Magma} that $G_0 \cap G_3 = G_{0,3}$. Hence $G_0$ is a string C-group and all string C-group representations with permutation representation graph $(F_1)$ are string C-groups. It remains to show that the four generators generate $A_n$. The element $\rho_0\rho_1$ is a 5-cycle and $G$ is primitive, as for instance $\rho_0$ cannot preserve any block system. Hence, by Theorem~\ref{GJ}, $G\cong A_n$.

The Schl\"afli type is obvious from the permutation representation graph.
\end{proof}
\begin{theorem}\label{25}
Let $n\equiv 2 \mod 4$ with $n\geq 10$.	
The group $A_n$ admits a string C-group representation of rank $5$, with Schl\"afli type $\{5, 5, 6, n-5\}$, with the following CPR-graph.
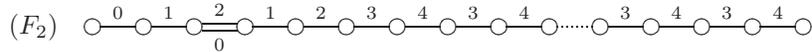
\begin{center}
\begin{tabular}{cc}

$(F_2)$&$ \xymatrix@-1.7pc{*+[o][F]{}  \ar@{-}[rr]^0 && *+[o][F]{}  \ar@{-}[rr]^1 &&*+[o][F]{} \ar@<.3ex>@{-}[rr]^2 \ar@<-.3ex>@{-}[rr]_0 && *+[o][F]{}  \ar@{-}[rr]^1 && *+[o][F]{}  \ar@{-}[rr]^2 && *+[o][F]{} \ar@{-}[rr]^{3} && *+[o][F]{}  \ar@{-}[rr]^{4} && *+[o][F]{} \ar@{-}[rr]^{3} && *+[o][F]{}  \ar@{-}[rr]^{4}   && *+[o][F]{} \ar@{.}[rr]  && *+[o][F]{} \ar@{-}[rr]^{3} &&*+[o][F]{} \ar@{-}[rr]^{4} &&*+[o][F]{}\ar@{-}[rr]^{3} &&*+[o][F]{} \ar@{-}[rr]^{4} &&*+[o][F]{}  }$\\

\end{tabular}
\end{center}
\end{theorem}
\begin{proof}
Let $G$ be a group having one of the permutation representation graphs above.
In this case, $G_4$ is  a sesqui-extension of one of the string C-groups representations for $(S_7\times 2)^+$ given in  Table~2 of~\cite{flm2}. Hence it is a string C-group.
By Propositon~\ref{0101tail} the group $G_0\cong A_{n-1}$.
The group $G_4\cong S_7$. Finally the group $G_{0,4}\cong S_6$.
One can easily check using {\sc Magma} or the atlas of regular polytopes~\cite{atlasl} that $G_4$ and $G_{0,4}$ are string C-group representations. So it remains to prove that $G_0$ is a string C-group to finish proving that all permutation representation graphs of $(F_2)$ give string C-groups.
The subgroup $G_{0,4}\cong S_6$ as stated above, $G_{0,1,4}\cong  D_{12}$ and $G_{0,1}\cong S_{n-4}$. 
Increasing $n$ will not change the intersection between $G_{0,1}$ and $G_{0,4}$. Hence we can check with {\sc Magma} that $G_{0,1} \cap G_{0,4} = G_{0,1,4}$ for $n=10$. Thus $G_{0,1}$ is a string C-group representation and so is $G_0$ and so is $G$ as $G_0\cong A_{n-1}$ and $G$ is transitive. Moreover $G\cong A_n$ since it is transitive on $n$ points and the stabilizer of a point in $G$ contains $G_0 \cong A_{n-1}$.

The Schl\"afli type is obvious from the permutation representation graph.
\end{proof}

\begin{theorem}\label{04}
Let $n\equiv 0 \mod 4$ with $n\geq 16$.	
The group $A_n$ admits a string C-group representation of rank $4$, with Schl\"afli type $\{3,12,LCM(n-8,6)\}$, with the following CPR-graph.
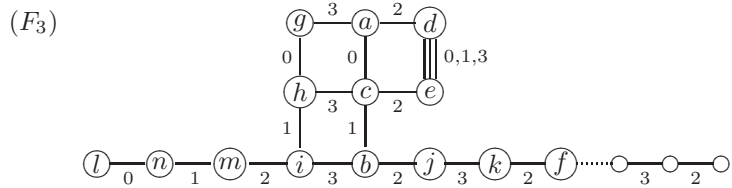
\begin{center}
\begin{tabular}{cc}
$(F_3)$
&$\xymatrix@-1.7pc{
&& && &&*+[o][F]{g}  \ar@{-}[rr]^3 &&  *+[o][F]{a}  \ar@{-}[rr]^2&& *+[o][F]{d}\\
 && && && && && & && && &&\\
&& && &&*+[o][F]{h}  \ar@{-}[rr]_3\ar@{-}[uu]^{0} &&  *+[o][F]{c} \ar@{-}[rr]_2\ar@{-}[uu]^{0} && *+[o][F]{e}  \ar@<-.5ex>@{-}[uu]_{0, 1, 3} \ar@<.5ex>@{-}[uu]^{}\ar@{-}[uu]^{}\\
 && && && && && & && && &&\\
*+[o][F]{l}  \ar@{-}[rr]_0  && *+[o][F]{n} \ar@{-}[rr]_1 && *+[o][F]{m}  \ar@{-}[rr]_{2} &&*+[o][F]{i}  \ar@{-}[rr]_{3}\ar@{-}[uu]^{1} &&*+[o][F]{b}  \ar@{-}[rr]_{2}\ar@{-}[uu]^{1} && *+[o][F]{j}  \ar@{-}[rr]_{3}  && *+[o][F]{k}  \ar@{-}[rr]_{2} && *+[o][F]{f}  \ar@{.}[rr]  && *+[o][F]{}  \ar@{-}[rr]_{3}  && *+[o][F]{}  \ar@{-}[rr]_{2} && *+[o][F]{}
} $\\
\end{tabular}
\end{center}
\end{theorem}
\begin{proof}
Let $G$ be a group having one of the permutation representation graphs above.
In this case, $G_3 \cong 2^2:S_3\times S_3$ and $G_{03}\cong D_{24}$ no matter what the value of $n$ is thanks to the shape of the graph. Observe that the left connected component of the graph, obtained when removing the 3-edges, gives the CPR graph of the octahedron. Thus it can easily be checked with {\sc Magma} that $G_3$ is a string C-group with type $\{3,12\}$.
The group $G_0$ is transitive on $n-1$ point, namely all vertices of the graph except $l$. Moreover, the stabilizer of $l$ and $n$ in $G$ has at most two more orbits thanks to the connected components of the permutation representation graph obtained by removing edges labelled 0 and 1. The element $(\rho_1\rho_2\rho_3\rho_2)^3$ moves point $i$ to point $d$ while fixing both $l$ and $n$. Hence $G_0$ is 2-transitive on $n-1$ vertices (all but $l$). Therefore $G_0$ is primitive on these points. Now, the element  $(\rho_1\rho_2\rho_3\rho_2) = (l)(n, j, m)(i,e,g,d,h)(a,c,f,b)...$ is such that the cycles we did not write are transpositions. Indeed, $\rho_1$ does not do anything on these points and so the action on these points is given by $\rho_2\rho_3\rho_2=\rho_3^{\rho_2}$ which is an involution. Hence $(\rho_1\rho_2\rho_3\rho_2)^{12} \in G_0$ is a 5-cycle fixing more than three points. By Theorem~\ref{GJ}, we can therefore conclude that $G_0\cong A_{n-1}$.
As $G_0$ is a simple group, since it is generated by three involutions (namely $\rho_1,\rho_2,\rho_3$), two of which commute, it is a string C-group by~\cite[Theorem 4.1]{CO}.
It remains to check that $G_{0,3} = G_0\cap G_3$ to prove that these graphs give indeed string C-groups. This can be checked with {\sc Magma} for $n=12$ and the result can be extended for any $n$.

The Schl\"afli type is obvious from the permutation representation graph.
\end{proof}
\begin{theorem}\label{05}
Let $n\equiv 0 \mod 4$ with $n\geq 12$.	
The group $A_n$ admits a string C-group representation of rank $5$, with Schl\"afli type $\{3,4,6,n-7\}$, with the following CPR-graph.
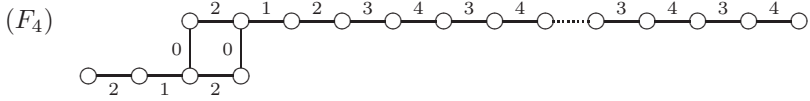
\begin{center}
\begin{tabular}{cc}
$(F_4)$
&$\xymatrix@-1.7pc{&& && *+[o][F]{}  \ar@{-}[rr]^{2} &&   *+[o][F]{}  \ar@{-}[rr]^{1}  && *+[o][F]{}  \ar@{-}[rr]^{2} && *+[o][F]{}  \ar@{-}[rr]^{3}  && *+[o][F]{}  \ar@{-}[rr]^{4} && *+[o][F]{}\ar@{-}[rr]^{3}  && *+[o][F]{}  \ar@{-}[rr]^{4} && *+[o][F]{}   \ar@{.}[rr]&& *+[o][F]{}\ar@{-}[rr]^{3}  && *+[o][F]{}  \ar@{-}[rr]^{4} && *+[o][F]{}\ar@{-}[rr]^{3}  && *+[o][F]{}  \ar@{-}[rr]^{4} && *+[o][F]{}\\
&& && && && && && & && && &&\\
*+[o][F]{} \ar@{-}[rr]_2 && *+[o][F]{}  \ar@{-}[rr]_{1} &&*+[o][F]{}  \ar@{-}[rr]_{2}\ar@{-}[uu]^{0} &&*+[o][F]{} \ar@{-}[uu]^{0}  } $\\
&\\
\end{tabular}
\end{center}
\end{theorem}
\begin{proof}
Let $G$ be a group having one of the permutation representation graphs above.
In this case, $G_4$ is a sesqui-extension of a string C-group isomorphic to $S_9$, that can be found for instance in the atlas~\cite{atlasl}. The group $G_{0,1}$ is a string C-group isomorphic to $S_{n-6}$ and  $G_{0,4}\cong S_5\times D_8$.  Now, as $\rho_2\rho_3$ is of order 6, $G_{0,1,4}\cong D_{12}$ and it is obvious from the permutation representation graph that $G_{0,4} \cap G_{0,1} = G_{0,1,4}$  and $G_{0,4} \cap G_{1,4} = G_{0,1,4}$. Hence $G_0$ and $G_4$ are string C-group by Proposition~\ref{arp}. As $G_0 \cap G_4$ must have orbits that are suborbits of those of $G_0$ and of those of $G_4$, we readily see that $G_0 \cap G_4 = G_{0,4}$. This concludes the proof that every graph of shape $(F_4)$ gives a string C-group. As $G$ is an even primitive group and  $(\rho_2\rho_3)^2$ is a $3$-cycle, we have that $G$ is isomorphic to $A_n$ by~Theorem~\ref{GJ}.

The Schl\"afli type is obvious from the permutation representation graph.
\end{proof}
\subsection{The odd case}

%
%
\begin{theorem}\label{1r}
Let $n\geq 12$ be an integer with $n\equiv 1 \mod 4$.
Let $4\leq r \leq (n-1)/2$.
The group $A_n$ admits a string C-group representation of rank $r$, with Schl\"afli type
$\{10, 3^{  \frac{n-1}{2}-2}\}$ when $r=  \frac{n-1}{2}$ and $\{10, 3^{r-4}, 6, \frac{n-1}{2}-r+3\}$ when $r<  \frac{n-1}{2}$, and with the following CPR graph.

$$ 
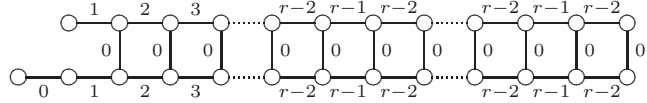
\xymatrix@-1.7pc{&& *+[o][F]{}  \ar@{-}[rr]^1 && *+[o][F]{}  \ar@{-}[rr]^2 && *+[o][F]{}  \ar@{-}[rr]^3 && *+[o][F]{}   \ar@{.}[rr] && *+[o][F]{}  \ar@{-}[rr]^{r-2} && *+[o][F]{}  \ar@{-}[rr]^{r-1} && *+[o][F]{}   \ar@{-}[rr]^{r-2} && *+[o][F]{} \ar@{.}[rr] && *+[o][F]{}  \ar@{-}[rr]^{r-2} && *+[o][F]{}  \ar@{-}[rr]^{r-1} && *+[o][F]{}   \ar@{-}[rr]^{r-2} && *+[o][F]{}\\
&& && && && && && && && && && && &&\\
 *+[o][F]{}  \ar@{-}[rr]_0 && *+[o][F]{}  \ar@{-}[rr]_1 && *+[o][F]{}   \ar@{-}[uu]^0   \ar@{-}[rr]_2 && *+[o][F]{}  \ar@{-}[uu]^0   \ar@{-}[rr]_3 && *+[o][F]{}  \ar@{-}[uu]^0     \ar@{.}[rr] && *+[o][F]{}  \ar@{-}[uu]_0   \ar@{-}[rr]_{r-2} && *+[o][F]{}  \ar@{-}[uu]_0   \ar@{-}[rr]_{r-1} &&  *+[o][F]{} \ar@{-}[uu]_0  \ar@{-}[rr]_{r-2} && *+[o][F]{}   \ar@{-}[uu]_0  \ar@{.}[rr] && *+[o][F]{}  \ar@{-}[uu]_0   \ar@{-}[rr]_{r-2} && *+[o][F]{}  \ar@{-}[uu]_0   \ar@{-}[rr]_{r-1} &&  *+[o][F]{} \ar@{-}[uu]_0  \ar@{-}[rr]_{r-2} && *+[o][F]{}   \ar@{-}[uu]_0  }$$
\end{theorem}

\begin{proof}
Let $G$ be a group having one of the permutation representation graphs above.
Clearly $G$ is an even group and it must be primitive as $\rho_0$ cannot preserve a non-trivial block system. Let us prove that $G$ is isomorphic to $A_n$.
We have that $(\rho_0\rho_1)^2$ is a $5$-cycle, hence by Theorem~\ref{GJ}, the group $G\cong A_n$.
It remains to prove that $G$ satisfies the intersection property. 
We know that  for $n=13$, the gorup $G$ is a string C-group of rank 6 and Schl\"afli type $\{10,3,3,3,3\}$. It can be checked with {\sc Magma} that $G$ is also string C-group for $n=13$ and $r\in\{4,5\}$.
By induction we may assume that $G_{r-1}$ is a sesqui-extension of a string C-group. Hence by Proposition~\ref{sesqui}, the group $G_{r-1}$ satisfies the intersection property.
By the first line of Table~\ref{TT}, it is straightforward that $G_0$ is a string C-group. Finally, $G_{0,r-1}=G_0 \cap G_{r-1}\cong S_{r-1}$.
Using this technique, we have just constructed string C-group representations of rank $r$ for every $4\leq r\leq  \frac{n-1}{2}$. Their Schl\"afli types are $\{10, 3^{  \frac{n-1}{2}-2}\}$ when $r=  \frac{n-1}{2}$ and $\{10, 3^{r-4}, 6, \frac{n-1}{2}-r+3\}$ when $r<  \frac{n-1}{2}$.
\end{proof}

%
%
The following theorem gives the string C-groups of rank $r=(n-1)/2$ in the case where $n\equiv 3 \mod 4$.
\begin{theorem}\label{3rmax}~\cite{flm2}
Let $n\geq 15$ be an integer with $n\equiv 3 \mod 4$.
Let $r = (n-1)/2$.
The group $A_n$ admits a string C-group representation of rank $r$, with Schl\"afli type
$\{5, 5, 6, 3^{r-7},6,6,3\}$, and with the following CPR graph.

$$
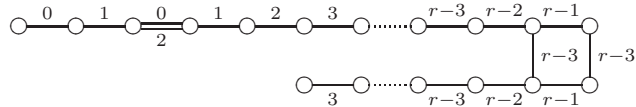
\xymatrix@-0.7pc{ *+[o][F]{}  \ar@{-}[r]^0 & *+[o][F]{}  \ar@{-}[r]^1 &*+[o][F]{}  \ar@{=}[r]^0_2 & *+[o][F]{} \ar@{-}[r]^1 & *+[o][F]{}  \ar@{-}[r]^2 & *+[o][F]{} \ar@{-}[r]^3 & *+[o][F]{} \ar@{.}[r] & *+[o][F]{}  \ar@{-}[r]^{r-3}& *+[o][F]{}  \ar@{-}[r]^{r-2}  & *+[o][F]{}  \ar@{-}[r]^{r-1}  & *+[o][F]{}\\
& & & & &*+[o][F]{}  \ar@{-}[r]_3 & *+[o][F]{}  \ar@{.}[r] & *+[o][F]{}  \ar@{-}[r]_{r-3} & *+[o][F]{}  \ar@{-}[r]_{r-2} & *+[o][F]{}  \ar@{-}[r]_{r-1}\ar@{-}[u]_{r-3}  & *+[o][F]{}\ar@{-}[u]_{r-3}} $$
\end{theorem}

From these examples, we construct examples of the same rank but for groups of degree $n+4k$ where $k$ is an integer by adding a sequence of alternating 0- and 1-edges of length $4k$ between the first and the second 2-edge (counting from the left).
\begin{theorem}\label{3r}
Let $n\geq 15$ be an integer with $n\equiv 3 \mod 4$.
Let $7 \leq r < (n-1)/2$.
The group $A_n$ admits a string C-group representation of odd rank $r$, with Schl\"afli type
$\{n-2(r-2), 12, 6, 3^{r-7},6,6,3\}$, and with the following CPR graph.

$$
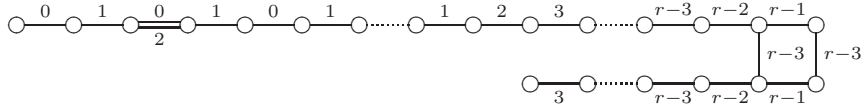
\xymatrix@-0.7pc{ *+[o][F]{}  \ar@{-}[r]^0 & *+[o][F]{}  \ar@{-}[r]^1 &*+[o][F]{}  \ar@{=}[r]^0_2 & *+[o][F]{} \ar@{-}[r]^1 &  *+[o][F]{}  \ar@{-}[r]^0 & *+[o][F]{}  \ar@{-}[r]^1 & *+[o][F]{}  \ar@{.}[r] & *+[o][F]{}  \ar@{-}[r]^1 &*+[o][F]{}  \ar@{-}[r]^2 & *+[o][F]{} \ar@{-}[r]^3 & *+[o][F]{} \ar@{.}[r] & *+[o][F]{}  \ar@{-}[r]^{r-3}& *+[o][F]{}  \ar@{-}[r]^{r-2}  & *+[o][F]{}  \ar@{-}[r]^{r-1}  & *+[o][F]{}\\
& & & & & & & & &*+[o][F]{}  \ar@{-}[r]_3 & *+[o][F]{}  \ar@{.}[r] & *+[o][F]{}  \ar@{-}[r]_{r-3} & *+[o][F]{}  \ar@{-}[r]_{r-2} & *+[o][F]{}  \ar@{-}[r]_{r-1}\ar@{-}[u]_{r-3}  & *+[o][F]{}\ar@{-}[u]_{r-3}} $$
\end{theorem}
\begin{proof}
Let $G$ be a group having the permutation representation graph above.
The group $G_0$ is acting as $S_{2(r-1)}$ on the orbit of size $2(r-1)$ and as $D_8$ on the orbit of size 4, making it isomorphic to  $A_{2(r-1)}:D_8$. Observe that $G_{0}$ has a structure that only depends on the rank, not on the degree of $G$.


The group $G_{0,r-1}$ is isomorphic to $S_{2(r-2)}:D_8$. It is a maximal subgroup of $G_0$. Hence
$G_0\cap G_{r-1} = G_{0,r-1}$. 

Let us now prove that $G_0$ and $G_{r-1}$ are string C-groups.
We start with $G_0$. The group $G_{0,1}$ is the same (up to removing the fixed points) as the one of Theorem~\ref{3rmax}. Hence it is a string C-group.
The group $G_{0,r-1}$ has the following permutation representation graph, where there might be more than one 1-edge disconnected from the rest of the graph.

$$
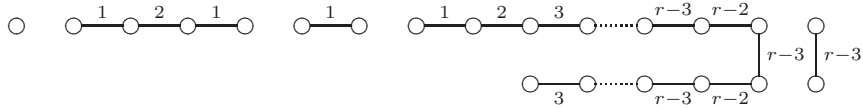
\xymatrix@-0.7pc{ *+[o][F]{}  \ar@{}[r] & *+[o][F]{}  \ar@{-}[r]^1 &*+[o][F]{}  \ar@{-}[r]^2 & *+[o][F]{} \ar@{-}[r]^1 &  *+[o][F]{}  \ar@{}[r] & *+[o][F]{}  \ar@{-}[r]^1 & *+[o][F]{}  \ar@{}[r] & *+[o][F]{}  \ar@{-}[r]^1 &*+[o][F]{}  \ar@{-}[r]^2 & *+[o][F]{} \ar@{-}[r]^3 & *+[o][F]{} \ar@{.}[r] & *+[o][F]{}  \ar@{-}[r]^{r-3}& *+[o][F]{}  \ar@{-}[r]^{r-2}  & *+[o][F]{}  \ar@{}[r]^{}  & *+[o][F]{}\\
& & & & & & & & &*+[o][F]{}  \ar@{-}[r]_3 & *+[o][F]{}  \ar@{.}[r] & *+[o][F]{}  \ar@{-}[r]_{r-3} & *+[o][F]{}  \ar@{-}[r]_{r-2} & *+[o][F]{}  \ar@{}[r]_{}\ar@{-}[u]_{r-3}  & *+[o][F]{}\ar@{-}[u]_{r-3}} $$

By Proposition~\ref{sesqui}, it suffices to prove that the following graph gives a string C-group to get that $G_{0,r-1}$ is also a string C-group. 

$$
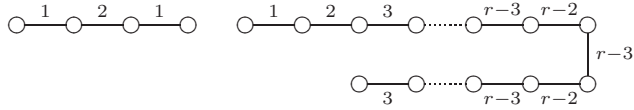
\xymatrix@-0.7pc{ *+[o][F]{}  \ar@{-}[r]^1 &*+[o][F]{}  \ar@{-}[r]^2 & *+[o][F]{} \ar@{-}[r]^1 &  *+[o][F]{}  \ar@{}[r] & *+[o][F]{}  \ar@{-}[r]^1 &*+[o][F]{}  \ar@{-}[r]^2 & *+[o][F]{} \ar@{-}[r]^3 & *+[o][F]{} \ar@{.}[r] & *+[o][F]{}  \ar@{-}[r]^{r-3}& *+[o][F]{}  \ar@{-}[r]^{r-2}  & *+[o][F]{}  \ar@{}^{}  \\
& & & & & &*+[o][F]{}  \ar@{-}[r]_3 & *+[o][F]{}  \ar@{.}[r] & *+[o][F]{}  \ar@{-}[r]_{r-3} & *+[o][F]{}  \ar@{-}[r]_{r-2} & *+[o][F]{}  \ar@{}\ar@{-}[u]_{r-3}  } $$

Let us call $H$ the group generated by this permutation representation graph. 
By Proposition~\ref{Sym3} the connected  component, of the graph above, on the right is a string C-group. By Proposition~\ref{sesqui} the graph that we obtain from the graph pictured above by removing the 2-edge on the left is a CPR graph.
Since removing the 2-edge on the left does not change the order of the group $H_1$, by~\cite[Proposition 2E17]{arp} we get that $H$ is a string C-group. Hence $G_0$ is a string C-group.

Let us now prove that $G_{r-1}$ is a string C-group.

The group $G_{r-2, r-1}$  is a sesqui-extension of the group $K$ having  the following permutation representation graph.
$$
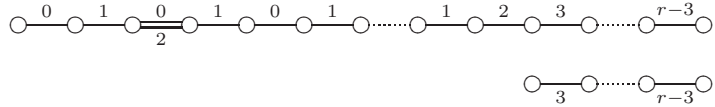
\xymatrix@-0.7pc{ *+[o][F]{}  \ar@{-}[r]^0 & *+[o][F]{}  \ar@{-}[r]^1 &*+[o][F]{}  \ar@{=}[r]^0_2 & *+[o][F]{} \ar@{-}[r]^1 &  *+[o][F]{}  \ar@{-}[r]^0 & *+[o][F]{}  \ar@{-}[r]^1 & *+[o][F]{}  \ar@{.}[r] & *+[o][F]{}  \ar@{-}[r]^1 &*+[o][F]{}  \ar@{-}[r]^2 & *+[o][F]{} \ar@{-}[r]^3 & *+[o][F]{} \ar@{.}[r] & *+[o][F]{}  \ar@{-}[r]^{r-3}& *+[o][F]{}  \\
& & & & & & & & &*+[o][F]{}  \ar@{-}[r]_3 & *+[o][F]{}  \ar@{.}[r] & *+[o][F]{}  \ar@{-}[r]_{r-3} & *+[o][F]{} } $$
Let  $a$ and $b$ be the sizes of the connected components of the graph above. For $r=6$, $K$ is a sesqui-extension of the C-group of Proposition~\ref{Sym4}, hence by Proposition~\ref{sesqui} it is a C-group isomorphic to $S_a\cong (S_a\times 2)^+$. By induction we may assume that $K_{r-3}$ is a C-group isomorphic to $(S_{a-1}\times S_{b-1})^+$. As $K_0$ is a string C-group and $K_0\cap K_{r-3}=K_{0,r-3}$ we have that 
$K$ is itself a string C-group. Moreover it is clearly isomorphic to $(S_a\times S_b)^+$. With this, using Proposition~\ref{sesqui}, we have that $G_{r-2, r-1}$ is a string C-group. Finally  $G_{0,r-1}\cap G_{r-2,r-1}\leq (D_8\times S_{2(r-3)}\times 2)^+\cong G_{0,r-2,r-1}$.


Hence we have proved that $G_{r-1}$ is a string C-group and therefore that $G$ itself is a string C-group.

It is easy to see from the permutation representation graph in the theorem that the Schl\"afli type of the string C-group of rank $r$ for $A_n$ obtained by this construction is
$\{n-2(r-2), 12, 6, 3^{r-7},6,6,3\}$.
\end{proof}

The previous two theorems permit us to construct examples of all possible odd ranks at least 7 for $A_n$ with $n\equiv 3 \mod 4$ and $n\geq 15$. We now construct an example of rank $(n-3)/2$ for $A_n$ from the example of rank $(n-1)/2$, that we will use to construct all examples of even rank at least 8.

\begin{theorem}\label{3rmax2}
Let $n\geq 19$ be an integer with $n\equiv 3 \mod 4$.
Let $r = (n-1)/2-1$.
The group $A_n$ admits a string C-group representation of rank $r$, with Schl\"afli type
$\{5, 5, 6, 3^{r-8},6,6,6,4\}$, and with the following CPR graph.

$$
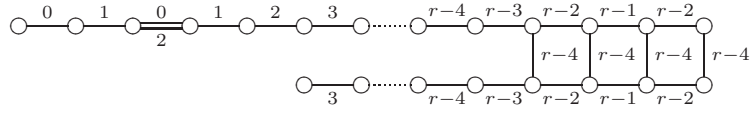
\xymatrix@-0.7pc{ *+[o][F]{}  \ar@{-}[r]^0 & *+[o][F]{}  \ar@{-}[r]^1 &*+[o][F]{}  \ar@{=}[r]^0_2 & *+[o][F]{} \ar@{-}[r]^1 & *+[o][F]{}  \ar@{-}[r]^2 & *+[o][F]{} \ar@{-}[r]^3 & *+[o][F]{} \ar@{.}[r] & *+[o][F]{}  \ar@{-}[r]^{r-4}& *+[o][F]{}  \ar@{-}[r]^{r-3}  & *+[o][F]{}  \ar@{-}[r]^{r-2}  & *+[o][F]{}  \ar@{-}[r]^{r-1}  & *+[o][F]{}  \ar@{-}[r]^{r-2}  & *+[o][F]{}  \\
& & & & &*+[o][F]{}  \ar@{-}[r]_3 & *+[o][F]{}  \ar@{.}[r] & *+[o][F]{}  \ar@{-}[r]_{r-4} & *+[o][F]{}  \ar@{-}[r]_{r-3} &  *+[o][F]{}  \ar@{-}[r]_{r-2}\ar@{-}[u]_{r-4}  &*+[o][F]{}  \ar@{-}[r]_{r-1}\ar@{-}[u]_{r-4} &  *+[o][F]{}  \ar@{-}[r]_{r-2}\ar@{-}[u]_{r-4}  & *+[o][F]{}\ar@{-}[u]_{r-4}} $$
\end{theorem}
\begin{proof}
Let $G$ be the group having the permutation representation graph in the statement of the theorem.
The group $G_{r-1}$ is a sesqui-extension of the group given in Theorem~\ref{3rmax}.
Hence it is a string C-group.
The group $G_0$ can be proved to be a string C-group using similar techniques as in the proof of the previous theorem.
The fact that $G_0 \cap G_{r-1} = G_{0,r-1}$ follows from the fact that $G_{r-1}$ is a sesqui-extension of the group given in Theorem~\ref{3rmax} and the orbits of the respective subgroups.
\end{proof}

As in the case of odd ranks, from these examples, we construct examples of the same rank but for groups of degree $n+4k$ where $k$ is an integer by adding a sequence of alternating 0- and 1-edges of length $4k$ between the 1-edge and the second 2-edge (counting from the left).

\begin{theorem}\label{3rmax3}
Let $n\geq 19$ be an integer with $n\equiv 3 \mod 4$.
Let $8 \leq r < (n-1)/2-1$.
The group $A_n$ admits a string C-group representation of even rank $r$, with Schl\"afli type
$\{n-2(r-1), 12, 6, 3^{r-8},6,6,6,4\}$, and with the following CPR graph.

$$
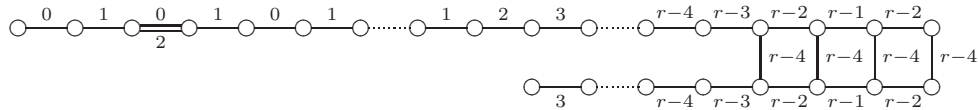
\xymatrix@-0.7pc{ *+[o][F]{}  \ar@{-}[r]^0 & *+[o][F]{}  \ar@{-}[r]^1 &*+[o][F]{}  \ar@{=}[r]^0_2 & *+[o][F]{} \ar@{-}[r]^1 &  *+[o][F]{}  \ar@{-}[r]^0 & *+[o][F]{}  \ar@{-}[r]^1 & *+[o][F]{}  \ar@{.}[r] & *+[o][F]{} \ar@{-}[r]^1 & *+[o][F]{}  \ar@{-}[r]^2 & *+[o][F]{} \ar@{-}[r]^3 & *+[o][F]{} \ar@{.}[r] & *+[o][F]{}  \ar@{-}[r]^{r-4}& *+[o][F]{}  \ar@{-}[r]^{r-3}  & *+[o][F]{}  \ar@{-}[r]^{r-2}  & *+[o][F]{}  \ar@{-}[r]^{r-1}  & *+[o][F]{}  \ar@{-}[r]^{r-2}  & *+[o][F]{}  \\
& & & & & & & & &*+[o][F]{}  \ar@{-}[r]_3 & *+[o][F]{}  \ar@{.}[r] & *+[o][F]{}  \ar@{-}[r]_{r-4} & *+[o][F]{}  \ar@{-}[r]_{r-3} &  *+[o][F]{}  \ar@{-}[r]_{r-2}\ar@{-}[u]_{r-4}  &*+[o][F]{}  \ar@{-}[r]_{r-1}\ar@{-}[u]_{r-4} &  *+[o][F]{}  \ar@{-}[r]_{r-2}\ar@{-}[u]_{r-4}  & *+[o][F]{}\ar@{-}[u]_{r-4}} $$
\end{theorem}
There are two ways to prove this theorem, either by a proof similar to that of Theorem~\ref{3r} or by a proof similar to that of Theorem~\ref{3rmax2}. We therefore leave the details to the interested reader.


\begin{theorem}\label{34}
Let $n\equiv 3 \mod 4$ with $n\geq 15$.	
The group $A_n$ admits a string C-group representation of rank $4$, with Schl\"afli type $\{10,7,4\}$ for $n=15$ and $\{2(n-10),14,4\}$ for $n>15$, with the following CPR-graph.
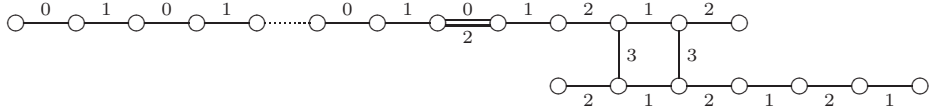
\begin{center}
\begin{tabular}{cc}
$$\xymatrix@-0.6pc{*+[o][F]{}  \ar@{-}[r]^0 &*+[o][F]{}  \ar@{-}[r]^1 &*+[o][F]{}  \ar@{-}[r]^0 &*+[o][F]{}  \ar@{-}[r]^1 &*+[o][F]{}  \ar@{.}[r] &*+[o][F]{}  \ar@{-}[r]^0 & *+[o][F]{}  \ar@{-}[r]^1 &*+[o][F]{}  \ar@{=}[r]^0_2 & *+[o][F]{}  \ar@{-}[r]^1 & *+[o][F]{}  \ar@{-}[r]^2 & *+[o][F]{}  \ar@{-}[r]^{1} & *+[o][F]{}  \ar@{-}[r]^2 & *+[o][F]{}&\\
&&&& & & & & & *+[o][F]{}  \ar@{-}[r]_{2} & *+[o][F]{}  \ar@{-}[r]_{1}\ar@{-}[u]_{3}  & *+[o][F]{}\ar@{-}[r]_{2}\ar@{-}[u]_{3} & *+[o][F]{}  \ar@{-}[r]_1& *+[o][F]{}  \ar@{-}[r]_2& *+[o][F]{}  \ar@{-}[r]_1& *+[o][F]{}  }$$
\end{tabular}
\end{center}
\end{theorem}
\begin{proof}
The group $G_0\cong 2^6:A_7:2$ for $n=15$ and $2^6:A_7:2\times 2$ for $n\geq 19$, no matter how big $n$ is.
It can easily be checked with {\sc Magma} that $G_0$ is a string C-group for $n=15$ and $n=19$ and since adding more points to the graph will not change the structure of $G_0$, we can conclude that $G_0$ is a string C-group for every $n\geq 15$.
The group $G_3$ acts as $S_{n-7}$ on the vertices of the top of the graph and acts as $D_{14}$ on the remaining vertices and it is a subgroup of $(A_{n-7}\times D_{14})^+$. We can thus conclude that $G_3\cong A_{n-7}\times D_{14}$. The group $G_{0,3}\cong D_{14}$ for $n=15$ and $C_2\times D_{14}$ when $n\geq 19$ (as there are extra 1-edges in the graph).
The group $G_{2,3}\cong D_{2(n-10)}$. It is obvious from the permutation representation graph that $G_{0,3}\cap G_{2,3} \cong C_2$.
Hence, by Proposition~\ref{arp}, the group $G_3$ is a string C-group.
Now, the intersection $G_{0}\cap G_{3} = G_{0,3}$ need only to be checked in the cases $n\in \{15,19\}$, which can be done with {\sc Magma}. Hence, again, by Proposition~\ref{arp}, we have that $G$ is a string C-group. 

It remains to show that $G\cong A_n$. The structure of $G_3$ shows that the action of $G_3$ on the $(n-7)$ vertices at the top of the graph is $A_{n-7}$. Hence there exists a cycle of order 3 in $G_0$ acting on those vertices. 
This cycle necessarily fixes the 7 other vertices, so it is a cycle of $G$. 
Moreover, that action is $(n-9)$-transitive on the top vertices. Hence the stabilizer, in $G$, of the leftmost vertex of the graph must be transitive on the remaining vertices and $G$ is 2-transitive, therefore primitive.
Then, by Theorem~\ref{GJ}, we can conclude that $G\geq A_n$. Since all generators of $G$ are even permutations, we must have $G\cong A_n$.

The Schl\"afli type follows immediately from the permutation representation graph.
\end{proof}

\begin{theorem}\label{35}
Let $n\equiv 3 \mod 4$ with $n\geq 15$.	
The group $A_n$ admits a string C-group representation of rank $5$, with Schl\"afli type $\{n-10,6,6,5\}$, with the following CPR-graph.
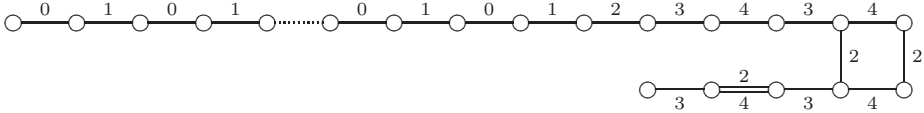
\begin{center}
\begin{tabular}{cc}
$$\xymatrix@-0.5pc{*+[o][F]{}  \ar@{-}[r]^0 &*+[o][F]{}  \ar@{-}[r]^1 &*+[o][F]{}  \ar@{-}[r]^0 &*+[o][F]{}  \ar@{-}[r]^1 &*+[o][F]{}  \ar@{.}[r] &*+[o][F]{}  \ar@{-}[r]^0 & *+[o][F]{}  \ar@{-}[r]^1 &*+[o][F]{}  \ar@{-}[r]^0 & *+[o][F]{}  \ar@{-}[r]^1 & *+[o][F]{}  \ar@{-}[r]^2 & *+[o][F]{}  \ar@{-}[r]^3 & *+[o][F]{_{}}  \ar@{-}[r]^{4}& *+[o][F]{_{}}  \ar@{-}[r]^{3}  & *+[o][F]{}  \ar@{-}[r]^{4}  & *+[o][F]{}\\
&&&& & & & & & & *+[o][F]{}  \ar@{-}[r]_3 & *+[o][F]{}  \ar@{=}[r]_4^2 & *+[o][F]{}  \ar@{-}[r]_{3} & *+[o][F]{}  \ar@{-}[r]_{4}\ar@{-}[u]_{2}  & *+[o][F]{}\ar@{-}[u]_{2}} $$

\end{tabular}
\end{center}
\end{theorem}
\begin{proof}
The group $G_0\cong S_{12}$ no matter how large $n$ is. One can easily check with {\sc Magma} that the permutation representation graph corresponding to $G_0$ is a CPR graph.
The group $G_{0,4}\cong 2^3:S_3\times S_3$ no matter how large $n$ is. 
 The group $G_{3,4}\cong S_{n-9}$ by Theorem~\ref{GJ}, as it contains a cycle of length 3, namely $(\rho_1\rho_2)^2$ and is obviously 2-transitive on $n-9$ vertices. Moreover, by \cite[Theorem 4.1]{CO}, it a string C-group as it is generated by three involutions, two of which commute.
 The group $G_{0,3,4}\cong D_{12}$. Looking at the respective orbits of $G_{0,4}$ and $G_{3,4}$ we can conclude that $G_{0,4}\cap G_{3,4} = G_{034}$ and therefore $G_4$ is a string C-group.
 Moreover, one can check that the group $G_4\cong A_{n-8}\times 2:S_{3}$ but this is not needed to finish the proof.
 Now, it is easy to check with {\sc Magma} that $G_0\cap G_4 = G_{0,4}$ for $n=15$ and this intersection does not depend on the degree of $G$. Therefore, by Proposition~\ref{arp}, we may conclude that $G$ is a string C-group with the given permutation representation graph.
A similar argument as in the proof of Theorem~\ref{34} permits to show that $G\cong A_n$.
The Schl\"afli type follows immediately from the permutation representation graph.
\end{proof}

\begin{theorem}\label{36}
Let $n\equiv 3 \mod 4$ with $n\geq 15$.	
The group $A_n$ admits a string C-group representation of rank $6$, with Schl\"afli type $\{n-10,6,3,5,3\}$, with the following CPR-graph.
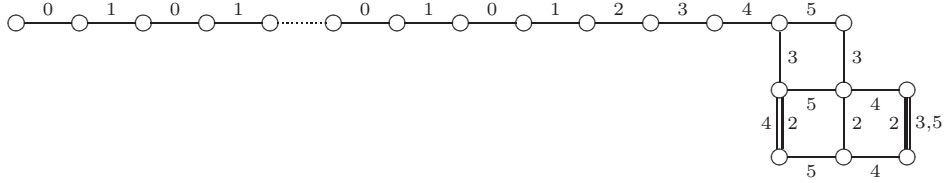
\begin{center}
\begin{tabular}{cc}
$$\xymatrix@-0.5pc{*+[o][F]{}  \ar@{-}[r]^0 &*+[o][F]{}  \ar@{-}[r]^1 &*+[o][F]{}  \ar@{-}[r]^0 &*+[o][F]{}  \ar@{-}[r]^1 &*+[o][F]{}  \ar@{.}[r] &*+[o][F]{}  \ar@{-}[r]^0 & *+[o][F]{}  \ar@{-}[r]^1 &*+[o][F]{}  \ar@{-}[r]^0 & *+[o][F]{}  \ar@{-}[r]^1 & *+[o][F]{}  \ar@{-}[r]^2 & *+[o][F]{}  \ar@{-}[r]^3 & *+[o][F]{_{}}  \ar@{-}[r]^{4}& *+[o][F]{_{}}  \ar@{-}[r]^{5} & *+[o][F]{}\\
& & &&&& & & & & & &*+[o][F]{} \ar@{-}[u]_{3} \ar@{-}[r]_5 & *+[o][F]{}  \ar@{-}[u]_{3}\ar@{-}[r]_4 & *+[o][F]{} \\
& & & &&&& & & & & &*+[o][F]{} \ar@{=}[u]_{2}^4 \ar@{-}[r]_5 & *+[o][F]{}  \ar@{-}[u]_{2}\ar@{-}[r]_4 & *+[o][F]{} \ar@{=}[u]^{2}\ar@{-}[u]_{3,5}
} $$

\end{tabular}
\end{center}
\end{theorem}

\begin{proof}
The group $G_0\cong S_{12}$ no matter how big $n$ is. One can easily check with {\sc Magma} that the permutation representation graph corresponding to $G_0$ is a CPR graph.
The group $G_{0,5}\cong S_7\times A_5$ no matter how big $n$ is. 
The group $G_{3,4,5}\cong S_{n-9}$ as proven in the previous theorem (for $G_{34}$ in the previous theorem is the same group as $G_{3,4,5}$ here).
Similarly, $G_{0,4,5} \cong 2^2:S_3\times S_3$.
As $G_{3,4,5}\cap G_{0,4,5}=G_{0,3,4,5}$ independently on how big $n$ is, we can conclude by Proposition~\ref{arp} that $G_{4,5}$ is a string C-group.
Similarly, as $G_{0,5}\cap G_{4,5} = G_{0,4,5}$ no matter how big $n$ is, we can conclude by Proposition~\ref{arp} that $G_5$ is a string C-group.
Finally, as $G_0\cap G_5 =  G_{0,5}$ no matter how big $n$ is, we conclude that $G$ is a string C-group.

It remains to show that $G\cong A_n$. Similar arguments as in the proof of the previous two theorems lead to that conclusion.
The Schl\"afli type follows immediately from the permutation representation graph.
\end{proof}
Observe that this last family of string C-groups of rank 6 gives, using the same general construction we used in Theorems~\ref{evenr} and~\ref{1r}, a family of string C-groups of rank 5 with Schl\"afli type $\{n-10,6,5,3\}$. 

\begin{theorem}\label{35bis}
Let $n\equiv 3 \mod 4$ with $n\geq 15$.	
The group $A_n$ admits a string C-group representation of rank $5$, with Schl\"afli type $\{n-9,6,5,3\}$, with the following CPR-graph.
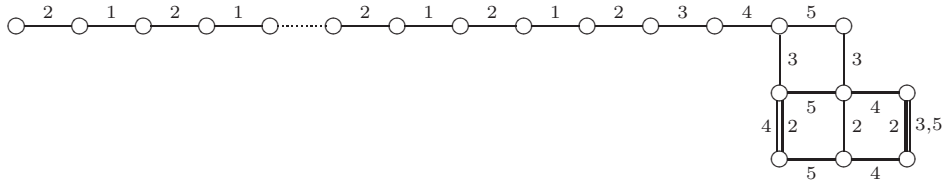
\begin{center}
\begin{tabular}{cc}
$$\xymatrix@-0.5pc{*+[o][F]{}  \ar@{-}[r]^2 &*+[o][F]{}  \ar@{-}[r]^1 &*+[o][F]{}  \ar@{-}[r]^2 &*+[o][F]{}  \ar@{-}[r]^1 &*+[o][F]{}  \ar@{.}[r] &*+[o][F]{}  \ar@{-}[r]^2 & *+[o][F]{}  \ar@{-}[r]^1 &*+[o][F]{}  \ar@{-}[r]^2 & *+[o][F]{}  \ar@{-}[r]^1 & *+[o][F]{}  \ar@{-}[r]^2 & *+[o][F]{}  \ar@{-}[r]^3 & *+[o][F]{_{}}  \ar@{-}[r]^{4}& *+[o][F]{_{}}  \ar@{-}[r]^{5} & *+[o][F]{}\\
& & &&&& & & & & & &*+[o][F]{} \ar@{-}[u]_{3} \ar@{-}[r]_5 & *+[o][F]{}  \ar@{-}[u]_{3}\ar@{-}[r]_4 & *+[o][F]{} \\
& & & &&&& & & & & &*+[o][F]{} \ar@{=}[u]_{2}^4 \ar@{-}[r]_5 & *+[o][F]{}  \ar@{-}[u]_{2}\ar@{-}[r]_4 & *+[o][F]{} \ar@{=}[u]^{2}\ar@{-}[u]_{3,5}
} $$

\end{tabular}
\end{center}
\end{theorem}

We leave the proof of this last theorem to the interested reader as it is very similar to the previous proofs.

\section{ Acknowledgements}

The authors thank Mark Mixer for observing that there was a mistake somewhere in the case where $n\equiv 3 \mod 4$ in a previous version of this paper.
This research was supported by the Portuguese
Foundation for Science and Technology (FCT- Funda\c c\~ao para a Ci\^encia e a Tecnologia),
through CIDMA - Center for Research and Development in Mathematics and
Applications, within project UID/MAT/04106/2013.


\bibliographystyle{amsplain}

\end{document}